\documentclass{amsart} \usepackage{url}
\usepackage{amsmath,amsfonts,euscript,amsthm,amssymb,amscd}
\usepackage[T1]{fontenc}
\usepackage{lmodern}
\usepackage{epigraph}
\usepackage[utf8]{inputenc}
\usepackage{enumerate}
\usepackage{tikz}

\theoremstyle{definition}
\newtheorem{Def}{Definition}[subsection]
\newtheorem{Def-Prop}[Def]{Definition-Proposition}

\newtheorem{Th}[Def]{Theorem}

\newtheorem{remark}[Def]{Remark}
\newtheorem{Prop}[Def]{Proposition}
\newtheorem{Lemma}[Def]{Lemma}
\newtheorem{Cor}[Def]{Corollary}

\newtheorem{Th1}{Theorem}

\DeclareMathOperator*{\hlim}{\text{holim}}

\DeclareMathOperator*{\clim}{\text{colim}}

\newcommand{\hdot}{{\:\raisebox{3pt}{\text{\circle*{1.5}}}}}
\newcommand{\hdotc}{{\:\raisebox{1pt}{\text{\circle*{1.5}}}}}

 \newcommand{\Ba}{\begin{array}}
 \newcommand{\Ea}{\end{array}}
\usepackage[all]{xy}
\usepackage{tikz}
 \tikzstyle{int}=[circle, draw,fill=black,outer sep=0,minimum size=3pt, inner sep=0]
  \tikzstyle{ext}=[circle, draw=black,outer sep=0,inner sep=1pt]
\input{diag}
\tikzset{snakeit/.style={decorate, decoration={snake, amplitude=.2mm,segment length=1mm}}}

\tikzset{ext/.style={circle, draw,inner sep=1pt}, int/.style={circle,draw,fill,inner sep=2pt},nil/.style={inner sep=1pt}}
\tikzset{cy/.style={circle,draw,fill,inner sep=2pt},scy/.style={circle,draw,inner sep=2pt},scyx/.style={draw,cross out,inner sep=2pt},scyt/.style={draw,regular polygon,regular polygon sides=3,inner sep=0.95pt}}
\tikzset{exte/.style={circle, draw,inner sep=3pt},inte/.style={circle,draw,fill,inner sep=3pt}}
\tikzset{diagram/.style={matrix of math nodes, row sep=3em, column sep=2.5em, text height=1.5ex, text depth=0.25ex}}
\tikzset{diagram2/.style={matrix of math nodes, row sep=0.5em, column sep=0.5em, text height=1.5ex, text depth=0.25ex}}
\tikzset{rowcolsep/.style={column sep=.2cm, row sep=.1cm}}

\tikzset{
  crossed/.style={
    decoration={markings,mark=at position .5 with {\arrow{|}}},
    postaction={decorate},
    shorten >=0.4pt}}

\tikzset{every picture/.style={baseline=-.65ex} }

\begin{document}

\title{Getzler-Kapranov complexes and moduli stacks of curves}
\author{Alexey Kalugin} 
\email{alexey.kalugin@msi.mpg.de}
\address{Max Planck Institut für Mathematik in den Naturwissenschaften, Inselstraße 22, 04103 Leipzig, Germany}
\begin{abstract} In this paper, we study so-called Getzler-Kapranov complexes and their relation to the cohomology of moduli stacks of curves.
\end{abstract}
\maketitle
\section{Introduction} 

\subsection{Introduction} We investigate an interplay between the cohomology of graph complexes and the cohomology $H_c^{\hdot}(\mathcal M_{g,n},\mathbb Q)$ of moduli stacks  $\mathcal M_{g,n}.$ This subject was initiated by S. Merkulov and T. Willwacher in \cite{MW} and later investigated by  M. Chan, S. Galatius, and S. Payne \cite{CGP1} \cite{CGP2} (see also \cite{AWZ} and \cite{AZ}). 
Recall that for any $g$ and $n$ such that $2g-2+n>0$ $\mathcal M_{g,n}$ is the separated, smooth and non-proper Deligne-Mumford stack \cite{DM} \cite{KN}. According to P. Deligne \cite{DelH} \cite{DelH2} the rational compactly supported cohomology $H^i_c(\mathcal M_{g,n},\mathbb Q)$ of $\mathcal M_{g,n}$ carries a weight filtration:
$$
W_0H^i(\mathcal M_{g,n},\mathbb Q)\subset \dots \subset W_{i-1}H^i(\mathcal M_{g,n},\mathbb Q)\subset H^i(\mathcal M_{g,n},\mathbb Q)
$$
The weight filtration reflects geometric and topological properties of moduli stacks $\mathcal M_{g,n}.$ A starting point of our study is a fascinating result by M. Chan, S. Galatius, and S. Payne \cite{CGP1} \cite{CGP2}. In \textit{ibid.}, it was shown that the weight zero quotient of $H_c^i(\mathcal M_{g,n},\mathbb Q)$ is identified with the cohomology of the $g$-loop part of an $n$-hairy (marked) graph complex:
$$
 W_0H_c^i(\mathcal M_{g,n},\mathbb Q)\cong H^i(\textsf B_g\textsf H_n\textsf {GC}).
$$
Here $\textsf B_g\textsf H_n\textsf {GC}_0^{\hdot}(\delta)$ is the combinatorial complex with cochains generated by at least trivalent graphs of genus $g$ with $n$-markings and a certain orientation. A differential $\delta$ is defined by splitting a vertex \cite{KWZ}, in particular when $n=0$ we recover a definition of the famous \textit{M. Kontsevich graph complex} $\textsf B_g\textsf {GC}^{\hdot}(\delta)$ \cite{Will}.

\subsection{Results} We extend the result of Chan-Galatius-Payne to higher weights. For every $g,n$ such that $2g-2+n>0$ we introduce a decorated graph complex $\textsf W_k\textsf  {GK}_{g,n}^{\hdot}$ which computes the weight $k$-quotient of $H_c^{\hdot}(\mathcal M_{g,n},\mathbb Q).$ These complexes are quasi-isomorphic to certain subcomplexes of a Getzler-Kapranov complex introduced in \cite{AWZ}.\footnote{In \textit{ibid.} this complex was defined as a covariant Feynman transform of the modular cooperad $C^{\hdot}(\overline{\mathcal M},\mathbb Q)$ from \cite{KG}.}
We call $\textsf W_k\textsf {GK}_{g,n}^{\hdot}$ a \textit{Getzler-Kapranov complex of the weight $k.$} Our first main result is:

\begin{Th1} For $g,n$ such that $2g-2+n>0$ we have the following description of the weight quotients on the compactly supported cohomology of $\mathcal                                                    M_{g,n}:$
$$
\mathrm {Gr}^W_k H^{i+k}_c(\mathcal M_{g,n},\mathbb Q)\overset{\sim}{\longrightarrow}  H^i(\textsf W_k\textsf{GK}_{g,n})
$$
\end{Th1}
In particular, we show that the weight zero Getzler-Kapranov complex $\textsf W_0\textsf {GK}_{g,n}^{\hdot}$ is quasi-isomorphic to $\textsf B_g\textsf H_n\textsf {GC}^{\hdot}(\delta)$ and we recover the result of Chan-Galatius-Payne. Further applying explicit computations of the rational cohomology of $\overline{\mathcal M}_{g,n}$ in low degrees from \cite{AC} we show that the weight quotients of $\mathcal M_{g,n}$ in degrees $1,3,5$ vanish and the first nontrivial weight quotient (after the weight zero) may appear in the weight two. Explicit computations, in this case, were made in the recent beautiful work of S. Payne and T. Willwacher \cite{PW1}. 
\par\medskip
The second object of our study is a \textit{hairy Getzler-Kapranov complex} $\textsf {H}_{\geq 1}\textsf {GK}_g^{\hdot}.$ Before we give a definition let us recall a construction from hairy graph complexes. Following \cite{TW} we consider a hairy graph complex $\textsf B_g \textsf H_{\geq 1} \textsf {GC}^{\hdot}(\delta+\chi).$ As a mere graded vector space this complex is defined by the rule:
$$
\textsf B_g \textsf H_{\geq 1} \textsf {GC}^{\hdot}:=\prod_{n=1}^{\infty} (\textsf B_g \textsf H_{n} \textsf {GC}^{\hdot}\otimes_{\Sigma_n} \mathrm {sgn}_n)^{\Sigma_n}
$$
A differential is defined as a sum $\delta+\chi,$ where $\chi$ is a differential acting by an inserting a hair in all possible ways:
$$
\chi\colon \Gamma \longmapsto \sum_{v\in V(\Gamma)} h_v(\Gamma)
$$
The result from \textit{ibid.} identifies the cohomology of this complex with the shifted cohomology of the M. Kontsevich graph complex: 
$$
\textsf B_g \textsf H_{\geq 1} \textsf {GC}^{\hdot}(\delta+\chi)\cong \textsf {B}_g\textsf {GC}^{\hdot-1}(\delta)
$$
\par\medskip 
A definition of the hairy Getzler-Kapranov $\textsf {H}_{\geq 1}\textsf {GK}_g^{\hdot}$ is reminiscent of the definition above. Cochains are collections of isomorphism classes of $C^{\hdot}(\overline{\mathcal M}_{g',n'},\mathbb Q)$-decorated stable graphs of genus $g$ with skew-symmetrised hairs (markings). A differential is defined as a sum of two differentials. The first differential splits a vertex and acts on decorations by a pullback along the clutching morphism and the second differential is defined by adding a decorated hair. Our first result about the hairy Getzler-Kapranov complexes (also proved in \cite{AWZ}) is the following:

\begin{Th1} For every $g\geq 2$ the hairy Getzler-Kapranov complex is quasi-isomorphic to the shifted compactly supported cochains of $\mathcal M_g$ with trivial coefficients:
$$
C_c^{\hdot}(\mathcal M_g,\mathbb Q)[-1]\overset{\sim}{\longrightarrow}\textsf {H}_{\geq 1}\textsf {GK}_g^{\hdot}
$$
\end{Th1}

Our second result (Conjecture $31$ from \cite{AWZ}) is a realisation of the hairy Getzler-Kapranov complex $\textsf {H}_{\geq 1}\textsf {GK}_g^{\hdot}$ as the total DG-vector space associated with the certain double complex:

\begin{Th1} Let $g\geq0$ then:
\par\medskip
\begin{enumerate}[(i)]
\item There exists a well-defined complex in the derived category of vector spaces:
$$
\begin{CD}
C_c^{\hdot}(\mathcal M_{g,1},\mathbb Q) @>{\nabla_1}>> \dots @>{\nabla_1}>> \left(C_c^{\hdot}(\mathcal M_{g,n},\mathbb Q)\otimes_{\Sigma_n} \mathrm {sgn}_n\right)^{\Sigma_n} @>{\nabla_1}>> \dots
\end{CD}
$$
Here $\nabla_1$ is a so-called \textit{Willwacher differential}:
$$
\nabla_1\colon \left(C_c^{\hdot}(\mathcal M_{g,n},\mathbb Q)\otimes_{\Sigma_n} \mathrm {sgn}_n\right)^{\Sigma_n}\longrightarrow \left(C_c^{\hdot}(\mathcal M_{g,n+1},\mathbb Q)\otimes_{\Sigma_{n+1}} \mathrm {sgn},_{n+1}\right)^{\Sigma_{n+1}}.
$$
\par\medskip
\item The induced morphism preservers the weight quotients on the compactly supported cohomology: 
$$\nabla_1\colon \left(\mathrm {Gr}_k^WH_c^{i}(\mathcal M_{g,n},\mathbb Q)\otimes_{\Sigma_n} \mathrm {sgn}_n\right)^{\Sigma_n}\longrightarrow \left(\mathrm {Gr}_k^WH_c^{i}(\mathcal M_{g,n+1},\mathbb Q)\otimes_{\Sigma_{n+1}} \mathrm {sgn}_{n+1}\right)^{\Sigma_{n+1}}.$$
Moreover, under the Chan-Galatius-Payne equivalence, the morphism ${\nabla_1}_{|_{W_0}}$ coincides with the differential $\chi.$ 
\par\medskip
  \item For $g\geq 2$ the total DG-vector space of the complex above is quasi-isomorphic to the hairy Getzler-Kapranov complex $\textsf {H}_{\geq 1}\textsf {GK}_g^{\hdot}.$

\end{enumerate}
\end{Th1}

Theorem $2$ and Theorem $3$ imply the following result:
\begin{Th1} For every $g\geq 2$ the total DG-vector space associated with the double complex:
\begin{equation*}
\begin{CD}
C_c^{\hdot}(\mathcal M_{g,1},\mathbb Q) @>{\nabla_1}>> \dots @>{\nabla_1}>> \left(C_c^{\hdot}(\mathcal M_{g,n},\mathbb Q)\otimes_{\Sigma_n} \mathrm {sgn}_n\right)^{\Sigma_n} @>{\nabla_1}>> \dots
\end{CD}
\end{equation*}
is quasi-isomorphic to $C_c^{\hdot}(\mathcal M_{g},\mathbb Q)[-1].$
\end{Th1}
Applying explicit computations of the cohomology of $\mathcal M_{1,n}$ from \cite{CF} \cite{Pet2} we get.
\begin{Th1} For $g=1$ the cohomology of the total DG-vector space associated with the double complex:
\begin{equation*}
\begin{CD}
C_c^{\hdot}(\mathcal M_{1,1},\mathbb Q) @>{\nabla_1}>> \dots @>{\nabla_1}>> \left(C_c^{\hdot}(\mathcal M_{1,n},\mathbb Q)\otimes_{\Sigma_n} \mathrm {sgn}_n\right)^{\Sigma_n} @>{\nabla_1}>> \dots
\end{CD}
\end{equation*}
is isomorphic to:
$$
\prod_{n=3}^{\infty} (S_{n+1}\oplus \overline{S}_{n+1}\oplus Eis_{n+1})[2n]\oplus \mathbb Q[3]. 
$$
Here $S_k$ (resp. $\overline{S}_k$) is a vector space of holomorphic (resp. antiholomorphic) cusp forms of the weight $k$ and $Eis_k$ is a vector space of the Eisenstein series of the weight $k$.\footnote{This complex is closely related to the cochain complex of an Artin stack $\mathcal M_1$ \cite{Tael}, however, we were unable to find a precise connection.}
\end{Th1}

\subsection{Methods} The main technical tools in this paper are P. Deligne's theory of mixed Hodge structures and constructible sheaves on moduli spaces of tropical curves. P. Deligne's theory is already classical and few words need to be said. Constructible sheaves on moduli spaces of tropical curves can be considered as a geometric avatar of Getzler-Kapranov's theory of modular operads (cf. \cite{GK}) and all our constructions can be painlessly translated to operadic language. We choose to work with constructible sheaves on moduli spaces of tropical curves to make our techniques more flexible (for example the differential of $\textsf {H}_{\geq 1}\textsf{GK}_g^{\hdot}$ is not covered by the Getzler-Kapranov formalism) and also to make a relation to moduli spaces of tropical curves (originated in \cite{CGP1} \cite{CGP2}) more transparent.

\subsection{Structure of the paper} In Section $2$ we recollect some facts about constructible sheaves on diagrams and mixed Hodge structures. In Section $3$ we recall some basic facts from the theory of moduli stacks $\mathcal M_{g,n}$ and the moduli spaces of tropical curves. In this section, we introduce the Getzler-Kapranov complex and reprove an original result of E. Getzler and M. Kapranov concerning the Feynman transform of the modular cooperad $C^{\hdot}(\overline{\mathcal M},\mathbb Q).$ Also in this section, we prove Theorem $1$ and further discuss the cohomology of the Getzler-Kapranov complexes $\textsf W_k\textsf {GK}^{\hdot}_{g,n}$ for small values of $k.$ In Section $4,$ we define the hairy Getzler-Kapranov complex and prove Theorem $2$ and Theorem $3.$ Further, we obtain Theorem $4$ as an immediate corollary and prove Theorem $5.$ 

\subsection{Acknowledgments} I would like to thank Nikita Markarian, Sergei Merkulov, and Thomas Willwacher for fruitful discussions. Special
thanks are due to Thomas Willwacher for introducing the author to almost all problems attacked in this paper and for his input on this work and to
Sergei Merkulov for his constant support and for bringing my attention to the beautiful paper \cite{MW}. Finally, I would like to thank Sergey Shadrin for various important suggestions and comments which helped the author to improve this text. This work was supported by the FNR project number: PRIDE $15/10949314/$GSM

\section{Preliminaries}

\subsection{Notation}

For a natural number $n\in \mathbb N_+$ we will denote by $[n]$ a finite set such that $[n]:=\{1,2,\dots,n\}.$ Let $I$ be a finite set, by $\mathrm {Aut}(I)$ we will denote a group of automorphisms of this set, in the case when $I=[n]$ we will use a notation $\Sigma_n:=\mathrm {Aut}([n])$ for a symmetric group on $n$-letters. We usually work over the field of rational number $\mathbb Q.$ For a $\mathbb Q$-linear representation $V$ of the finite group $G,$ we will denote by $V^G$ (resp. $V_G$) the space of $G$-invariants (resp. $G$-coinvariants). Since the characteristic of the field $\mathbb Q$ is zero we have a canonical isomorphism
$V_G\overset{\sim}{\longrightarrow} V^G$ and hence we will freely switch between invariants and coinvariants. For a finite $S$ we will denote by  $V\langle S\rangle$ a free $\mathbb Q$-vector space $V\langle S\rangle$ generated by the set $S,$ by $\det(S):=\bigwedge^{\dim V\langle S\rangle} V\langle S\rangle$ we denote a determinant of $V\langle S\rangle.$ By $\mathsf{Top}$ we denote a category of locally compact topological spaces equipped with a stratification and by $\mathsf {Cat}$ we denote a $2$-category of all categories. For categories $A$ and $B$ we denote by $\mathsf{Fun}(A,B)$ the corresponding category of functors. 

\subsection{Combinatorial sheaves}

Let $X$ be a topological space equipped with a Whitney stratification $\mathcal S:=\{S_{\alpha}\}.$ Then we have the corresponding category of $\mathcal S$-\textit{smooth combinatorial sheaves} on $X$ denoted by $\textsf{Sh}_{c}(X,\mathcal S).$ This is a full subcategory of the category of sheaves of finite-dimensional $\mathbb Q$-vector spaces on $X.$ By $\textsf {Ch}_c(X,\mathcal S)$ we denote the category of complexes of sheaves with constructible cohomology i.e. $H^{\hdot}(\EuScript K) \in\textsf{Sh}_{c}(X,\mathcal S).$ The corresponding triangulated category will be denoted by $\textsf D_c(X,\mathcal S).$ According to Proposition $1.9$ from \cite{KaSh} a natural comparison functor:
\begin{equation}\label{comp1}
\textsf {D(Sh}_c(X,\mathcal S))\longrightarrow \textsf D_c(X,\mathcal S)
\end{equation}
is the equivalence of triangulated categories. Objects of the category $\textsf D_c(X,\mathcal S)$ will be called \textit{$\mathcal S$-smooth combinatorial DG-sheaves on $X.$}
\par\medskip
Denote by $\textsf S$ a poset associated with the stratified space $(X,\mathcal S).$ By a definition, an object $\alpha$ in $\textsf S$ corresponds to a stratum $S_{\alpha}$ in $\mathcal S$ and for two objects $\alpha$ and $\beta$ in $\textsf S$ we define the arrow $\alpha \rightarrow \beta $ if $S_{\alpha}\subset \overline{S}_{\beta}.$ Hence we construct a functor:
\begin{equation}\label{comp2}
Real\colon \textsf {Sh}_c(X,\mathcal S)\longrightarrow \textsf {Fun}(\textsf S,{\textsf {Vect}}_{\mathbb Q}^f),
\end{equation}
which maps a combinatorial sheaf $\EuScript K$ to the functor $K\in \textsf{Fun}(\textsf S,\textsf {Vect}_{\mathbb Q}^f).$ This functor is defined by the following rule. A value of the functor $K$ at $\alpha$ equals $K_{\alpha}:=H^0(S_{\alpha},\EuScript K),$ and for every pair of adjusted strata $S_{\alpha}\subset \overline{S}_{\beta}$ we have the corresponding \textit{variation operator}:
$$
var_{\alpha,\beta}\colon K_{\alpha}\longrightarrow K_{\beta}.
$$
It is easy to see that $Real$ is the equivalence of abelian categories and hence from \eqref{comp1} we get the following:
\begin{Prop} The composition of the functors above induces the equivalence of the triangulated categories:
$$
\textsf D_c(X,\mathcal S)\overset{\sim}{\longrightarrow} \textsf D(\textsf {Fun(S,Vect}_{\mathbb Q}^f))
$$
\end{Prop}
We have a derived functor of the global sections with compact support:
$$
R\Gamma_c(X,-)\colon \textsf{D}_{c}(X,\mathcal S)\longrightarrow \textsf D(\textsf {Vect}_{\mathbb Q}^f).
$$
In terms of \eqref{comp2} the derived functor of global sections with compact support can be realised by the following complex \cite{GK}:

\begin{Prop}\label{GK}
The DG-vector space $R\Gamma_c(X,\EuScript K)$ is naturally quasi-isomorphic to (the total DG-vector space arising from) the complex:
$$
\bigoplus_{\dim S_{\alpha}=0}  K_{\alpha}\otimes \textsf {or}_{{\alpha}}\rightarrow \bigoplus_{\dim S_{\alpha}=1}  K_{\alpha}\otimes \textsf {or}_{{\alpha}}\rightarrow\dots
$$
Here $\textsf {or}_{\alpha}:=H_c^{\dim S_{\alpha}}(S_{\alpha},\mathbb Q)$ is a one-dimensional space of orientations and the sum over strata of dimension $m$ is placed in degree $m.$
\end{Prop}
Let $\EuScript K$ be a combinatorial DG-sheaf. For $k\in\mathbb Z$ we denote by $H^k(\EuScript K)\in \textsf {Sh}_c(X,\mathcal S)$ the corresponding $k$-cohomology of the combinatorial DG-sheaf.
\subsection{Sheaves on diagrams}

Let $C$ be a category, in this paper by a $C$-\textit{diagram} $X_{C}$ we will understand a functor $X_C\colon C\longrightarrow \mathsf {Top},$ such that for every $i\in C$ a slice category $C_i$ defines the stratification on $X_i,$ denoted by $\mathcal S_i.$ Following P. Deligne \cite{SD} (Exposé $\mathrm V^{bis}$) we have a notion of a sheaf on such objects. Let $X_C$ be a diagram represented by a system $\{X_i\}_{i\in C},$ where $f_{ij}\colon X_i\hookrightarrow X_j.$ Then using the construction from \textit{ibid.}, one may define a $2$-functor $C\longrightarrow \textsf {Cat}$ by sending each object $i\in C$ to the category of $\mathcal S_i$-smooth combinatorial sheaves $\textsf {Sh}_c(X_i,\mathcal S_i)$ and every morphism $f\colon i\longrightarrow j$ to a pullback functor $f^*_{ij}\colon \textsf {Sh}_c(X_j,\mathcal S_j)\longrightarrow \textsf {Sh}_c(X_i,\mathcal S_i).$ We can consider a category of lax sections of the corresponding Grothendieck fibration denoted by $\textsf {Sh}^{lax}_c(X_C,\mathcal S)$ \cite{Gro}. This is an abelian category, with the corresponding category of cocartesian sections will be denoted by $\textsf {Sh}_c(X_C,\mathcal S).$ Objects of the category $\textsf {Sh}^{lax}_c(X_C,\mathcal S)$ (resp. $\textsf {Sh}_c(X_C,\mathcal S)$) will be called $\mathcal S$-smooth lax combinatorial sheaves on the $C$-diagram $X_C$ (resp. $\mathcal S$-smooth combinatorial sheaves on the $C$-diagram $X_C).$ For the reader's convenience we will unpack these definitions:
\begin{Def}
An \textit{$\mathcal S$-smooth lax combinatorial sheaf $\EuScript P$ on the $C$-diagram $X_C$} is a combinatorial sheaf $\EuScript P_i$ on ${X_i}$ for every $i\in C,$ such that for every $m\colon i\longrightarrow j$ we have a morphism:
$$
\vartheta\colon m^*\EuScript P_j\longrightarrow \EuScript P_i
$$
which is compatible with compositions in $C.$
\par\medskip
We say that $\EuScript P$ is an \textit{$\mathcal S$-smooth combinatorial sheaf on the $C$-diagram $X_C$} if for all $m\colon i\longrightarrow j$ a morphism $\vartheta$ is the isomorphism.
\end{Def}
\textit{Mutatis mutandis} we define a $2$-functor $C\longrightarrow \textsf {Cat}$ which takes every object $i$ to the category of \textit{all sheaves} of finite-dimensional vector spaces $\textsf {Sh}(X_i).$ The corresponding category of lax sections will be denoted by $\textsf {Sh}(X_C).$ We denote by $\textsf D_{lax}(X_C)$ the associated derived category and by $\textsf D_c(X_C,\mathcal S)$ the triangulated subcategory with the cohomology in $\textsf {Sh}_c(X_C,\mathcal S).$ Objects of this category will be called $\mathcal S$-smooth combinatorial DG-sheaves on the $C$-diagram $X_C.$ Once again for the reader's convenience, we will unpack this definition:
\begin{Def}\label{DGSh}
An \textit{$\mathcal S$-smooth DG-combinatorial sheaf $\EuScript P$ on} $X_C$ is a complex of lax sheaves on $X_C$ such that for every $i\in C$ the cohomology of these sheaves are combinatorial and for every $m\colon i\longrightarrow j$ we have a quasi-isomorphism:
$$
\vartheta\colon m^*\EuScript P_j\overset{\sim}{\longrightarrow} \EuScript P_i
$$
which is compatible with compositions in $C.$
\end{Def}

Note that by the definition of the $C$-diagram $X_C$ for any morphism of stratified spaces $f_{ij}\colon (X_i,\mathcal S_i)\longrightarrow (X_j,\mathcal S_j)$ we have the induced functor between posets $F_{ij}\colon \textsf S_i\longrightarrow \textsf S_j.$ Analogous to the definitions above one may consider a $2$-functor $C\longrightarrow \textsf {Cat}$ which sends every $i\in C$ to the category $\textsf {Fun(S}_i,\textsf {Vect}_{\mathbb Q}^f)$ and every morphism $i\longrightarrow j$ to a restriction functor: 
$$Res(F_{ij})\colon \textsf {Fun(S}_j,\textsf {Vect}_{\mathbb Q}^f)\longrightarrow \textsf {Fun(S}_i,\textsf {Vect}_{\mathbb Q}^f).$$ 
The corresponding category of lax sections will be denoted by $\textsf {Sec}(C,\textsf {Fun}).$ By $\textsf {D}_{cocart}(\textsf {Sec}(C,\textsf {Fun}))$ we will denote the triangulated category with the cocartesian cohomology. Then we have the following:

\begin{Prop} The functor $Real$ induces an equivalence of triangulated categories:

\begin{equation}\label{comp5}
\textsf D_c(X_C,\mathcal S)\overset{\sim}{\longrightarrow} \textsf {D}_{cocart}(\textsf {Sec}(C,\textsf {Fun}))
\end{equation}

\end{Prop}

\begin{proof} The proof boils down to checking that under \eqref{comp2} the pullback functor $f_{ij}^*$ corresponds to the restriction functor $Res(F_{ij})$ which is a direct verification.

\end{proof}

Denote by $\widetilde{X}_C:=\clim_C X_C$ the corresponding colimit in the category of topological spaces. For every $i\in C$ we denote by $v_i\colon X_i\longrightarrow \widetilde{X}_C$ a canonical morphism.  By $\textsf D(\widetilde{X}_C)$ we will denote a derived category of sheaves of vector spaces on the topological space $\widetilde {X}.$ We have a natural comparison functor:
\begin{equation}\label{comp3}
Rv_*\colon\textsf D_c(X_C,\mathcal S)\longrightarrow \textsf D(\widetilde{X}_C).
\end{equation}
defined on objects by the following rule:
$$
Rv_*\colon \EuScript P\longmapsto \hlim_{i\in C} Rv_{i*}\EuScript P_i,\qquad \EuScript P\in\textsf D_c(X_C,\mathcal S).
$$
Note that the homotopy limit above is well defined since the category of sheaves is complete. One may show that \eqref{comp3} induces the fully faithful embedding of triangulated categories. By a \textit{constant sheaf on the $C$-diagram $X_C$} we understand an object $\underline{\mathbb Q}_{X_C}$ in the category $\textsf D_c(X_C,\mathcal S),$ defined as follows: for every $i\in C$ the value of $\underline{\mathbb Q}_{X_C}$ at $i$ is a one-dimensional constant sheaf on $X_i,$ and connecting morphisms are identical.
\par\medskip
By a morphism of $C$-diagrams $P\colon X_C\longrightarrow Y_C$ we understand a natural transformation between functors $X_C$ and $Y_C,$ such that every square of the natural transformation is fibered. If every morphism $P_i$ in the natural transformation $P$ is closed (resp. open) inclusion we call the morphism $P$ closed (resp. open) inclusion. For an inclusion $P\colon X_C\longrightarrow Y_C$ of $C$-diagrams we denote by $(Y\setminus X)_C$ the corresponding complement $C$-diagram defined by the obvious rule. Note that if $P$ is open (resp. closed) the corresponding inclusion $(Y\setminus X)_C\longrightarrow Y_C$ will be closed (resp. open). Due to the proper base change theorem, we have a functor of $!$-extension along open or (resp. closed) morphisms. The corresponding standard restriction functors are also well defined. Therefore if we have a sequence of morphisms of $C$-diagrams $j\colon U_C\hookrightarrow X_C,$ $\Delta\colon Z_C\hookrightarrow X_C,$ such that $j$ is an open inclusion and $\Delta$ is an inclusion of the corresponding closed complement we can apply the standard technique of Gysin triangles (cf. \cite{Beh}). That means that with a sequence of functors:
\begin{equation}\label{tr}
{j_!} \colon \textsf D_c(U_C,\mathcal S)\longleftrightarrow\textsf D_c(X_C,\mathcal S)\colon j^*,\qquad \Delta^*\colon\textsf D_c(X_C,\mathcal S) \longleftrightarrow\textsf D_c(Z_C,\mathcal S)\colon \Delta_*,
\end{equation}
we associate the standard distinguished triangle in the category $\textsf D_c(X_C,\mathcal S):$
$$
j_!j^*\longrightarrow \mathrm {Id}\longrightarrow \Delta_*\Delta^*\longrightarrow j_!j^*[1]\longrightarrow.
$$
\par\medskip
We define a derived functor of the global sections with compact support:
\begin{equation}\label{compa}
R\Gamma_c(X_C,-)\colon \textsf {D}_c(X_C,\mathcal S)\longrightarrow \textsf D(\textsf {Vect}_{\mathbb Q}),
\end{equation}
by the following rule:
$$
R\Gamma_c(X_C,\EuScript  P):=\hlim_{i\in C} R\Gamma_c(X_i,\EuScript P_i),
$$
where $\EuScript P\in \textsf D_{c}(X_C,\mathcal S)$ and connecting homomorphisms are defined by composition with derived direct image functors $Rf_{ij*}$ (these morphisms are well defined since $f_{ij}$ are closed morphisms and therefore a $!$-extension can be identified with the standard extension functor). By construction we have a canonical morphism:
\begin{equation}\label{comp4}
R\Gamma_c(\widetilde{X}_C,-)\longrightarrow R\Gamma_c(X_C,-).
\end{equation}

\subsection{P. Deligne's resolution} In this subsection, we recall some facts about the mixed Hodge structures that we will need.
\par\medskip
Let $\mathcal X$ be a smooth and separated Deligne-Mumford stack equipped with a compactification $\mathcal X\hookrightarrow \mathcal W,$ such that the complement    $\mathcal W\backslash \mathcal X:=\mathcal D$ is a divisor with normal crossings. Note that the divisor $\mathcal D$ is          naturally stratified $\{\mathcal D_p\}_{p\in \mathbb N}$, where a stratum $\mathcal D_p$ contains points of multiplicity at least $p.$ By             $\widetilde{\mathcal D}_p\rightarrow \mathcal D_p$ we will denote the corresponding (smooth) normalised stack. One may think about an element $z$      in $\widetilde{\mathcal D}_p$ as a point $x\in \mathcal D_p$ and $p$ components of $\mathcal D$ through $z.$ A set of these components will be         denoted by $E_p(z).$ We also set $\widetilde{\mathcal D}_0=\mathcal D_0=\mathcal W.$ Let $\widetilde{\mathcal D}_{p-1,p}$
be space whose points are pairs $(y,L),$ where $y=(x,E_p(y)) \in \widetilde{\mathcal D}_p$ and $L\in E_p(y).$ Note that we have a natural correspondence:
\begin{equation*}
\begin{diagram}[height=2.5em,width=2.5em]
\widetilde{\mathcal D}_{p-1,p} &   \rTo^{\xi_{p-1,p}}   &   \widetilde{\mathcal D}_{p-1}  &  \\
\dTo^{\pi_{p-1,p}} & &  && \\
\widetilde{\mathcal D}_p &      &    \\
\end{diagram}
\end{equation*}
Most of the time we will omit indices in the notation of morphisms above.
Note that by varying point $z$ in $\widetilde{\mathcal D}_p$ we obtain a local system $E_p$ on $\widetilde{\mathcal D}_p$ with the corresponding determinant local system $\varepsilon_p:=\det(E_p).$ We have the following isomorphism:
$$
\pi^!\varepsilon_p\cong \pi^*\varepsilon_p\otimes or_{\pi}\cong\xi ^*\varepsilon_{p-1}.
$$
Where by $or_{\pi}:=\pi^!\underline{\mathbb Q}_{\widetilde{\mathcal D}_p}$ we have denoted a relative dualising (orientation) sheaf. Therefore we can define a pull-push homomorphism:
\begin{equation}
\begin{diagram}[height=2.5em,width=2.5em]
C^{\hdot}(\widetilde{\mathcal D}_{p-1},\varepsilon_{p-1}) &   \rTo^{\xi^*}   & C^{\hdot}(\widetilde{\mathcal                                      D}_{p-1,p},\xi^*\varepsilon_{p-1})\cong C^{\hdot}(\widetilde{\mathcal D}_{p-1,p},\pi^!\varepsilon_{p}) &   \rTo^{\pi_!}   &                         C^{\hdot}(\widetilde{\mathcal D}_{p},\varepsilon_{p}) \\
\end{diagram}
\end{equation}
Complexes that compute the cohomology of $\widetilde{\mathcal D}_p$ with coefficients in $\varepsilon_p$ naturally form a DG-vector space:
\begin{equation}
\begin{diagram}[height=3em,width=3em]
C^{\hdot}(\mathcal W,\mathbb Q)&   \rTo^{\pi_!\xi^*}   & C^{\hdot}(\widetilde{\mathcal D}_1,\varepsilon_{1})&   \rTo^{\pi_!\xi^*}   & \dots &     \rTo^{\pi_!\xi^*}   & C^{\hdot}(\widetilde{\mathcal D}_p,\varepsilon_{p})&   \rTo^{\pi_!\xi^*}   &\dots
\end{diagram}
\end{equation}
The following Proposition should be well known to experts:
\begin{Prop}\label{D1}

The total DG-vector space of the complex above is quasi-isomorphic to $C_c^{\hdot}(\mathcal X,\mathbb Q).$

\end{Prop}

\begin{proof}
This proposition is Poincaré-Verdier dual to the original result of P. Deligne for the cohomology of $\mathcal W.$ For the details, we refer to P. Deligne's beautiful paper       \cite{DelH}.\footnote{The construction in \textit{ibid.} was given for complex varieties but can be generalised to the case of Deligne-Mumford stacks.}

\end{proof}

According to P. Deligne, the compactly supported cohomology $H_c^k(\mathcal X,\mathbb Q)$ of a smooth Deligne-Mumford stack $\mathcal X$ carries a canonical \textit{mixed Hodge structure} and hence the increasing \textit{weight filtration}:
$$
W_0H_c^k(\mathcal X,\mathbb Q) \subset \dots \subset W_{k-1}H_c^k(\mathcal X,\mathbb Q)\subset H_c^k(\mathcal X,\mathbb Q)
$$
Following P. Deligne one can use the DG-vector space above to define the weight filtration on $H^{\hdot}_c(\mathcal X,\mathbb Q).$ For $j\in \mathbb N$ denote by $C_{j}^{\hdot}$ the following cochain complex:
\begin{equation}
\begin{diagram}[height=3em,width=3em]
H^j(\mathcal W,\mathbb Q)&   \rTo^{\pi_!\xi^*}   & H^j(\widetilde{\mathcal D}_1,\varepsilon_1)&   \rTo^{\pi_!\xi^*}   & \dots &                       \rTo^{\pi_!\xi^*}   & H^j(\widetilde{\mathcal D}_p,\varepsilon_p) & \rTo^{\pi_!\xi^*}\dots
\end{diagram}
\end{equation}
Then we have:

\begin{Prop}\label{D2} For every smooth and separated Deligne-Mumford stack $\mathcal X,$ compactified by the divisor $\mathcal D$ with normal crossings we have the        following description of the weight filtration on $H_c^{\hdot}(\mathcal X,\mathbb Q):$
$$
H^i(C_j)\cong \mathrm {Gr}_j^WH^{i+j}_c(\mathcal X,\mathbb Q).
$$

\end{Prop}

\begin{proof}

Once again we refer to \cite{DelH}.
\end{proof}

\section{Sheaves on moduli spaces of tropical curves}

\subsection{Tropical curves}

By a graph $\Gamma$ we understand a connected graph with loops and parallel edges allowed. More precisely a (nonnecessary connected) graph $\Gamma$ is a triple $(H(\Gamma),\sigma_1,\sim)$ where $H(\Gamma)$ is a finite set called the set of \textit{half-edges} of a graph $\Gamma,$ a fixed point free involution $\sigma_1\colon H(\Gamma)\longrightarrow H(\Gamma)$ with a set of orbits $E(\Gamma):=H(\Gamma) / \sigma_1$ called the set of \textit{edges} of a graph $\Gamma$ and an equivalence relation $\sim$ on $H(\Gamma)$ with the corresponding set of equivalence classes denoted by     $V(\Gamma):=H(\Gamma)/ \sim$ and called the set of \textit{vertices} of $\Gamma.$ For a vertex $v\in V(\Gamma)$ we will denote by                         $H_{v}(\Gamma):=p^{-1}(v),$ a set of half-edges attached to the vertex $v,$ where $p\colon H(\Gamma)\longrightarrow V(\Gamma)$ is a canonical              projection. A number of elements in $H(v)$ will be called the \textit{valency} of the vertex $v$ and denoted by $val(v).$ It is possible to consider         $\Gamma$ as a one-dimensional CW complex and hence we will call $\Gamma$ a \textit{connected graph} if it is connected as a one-dimensional CW complex, further we assume that all graphs are connected. By a \textit{weighted $n$-marked graph} we understand a triple $(\Gamma,w,m)$ where $\Gamma$ is a graph together with function                           $w\colon V(\Gamma)\longrightarrow \mathbb N$ called the \textit{weight (genus) function} and the function $m\colon \{0,\dots n\}\longrightarrow V(\Gamma)$ called             $n$-\textit{marking}. The \textit{genus} of a weighted $n$-marked graph $(\Gamma,w,m)$ is defined by the standard formula:
 $$g:=b_1(\Gamma)+\sum_{v\in V(\Gamma)} w(v),$$ 
 where $b_1(\Gamma)=|E(\Gamma)|-|V(\Gamma)|+1$ is the first Betti number of $\Gamma,$ considered as a one dimensional CW complex. For a vertex $v$ in a weighted marked graph we will use the following notation $n_v:=val(v)+|m^{-1}(v)|.$ We say that a graph is \textit{stable} if for every vertex $v$ we have the following equality $2w(v)-2+n_v>0.$ Sometimes it will be convenient to consider markings          $m^{-1}(v)$ attached to a vertex as marked \textit{hairs} (legs) attached to $v.$ We will freely switch between these two notation.
\par\medskip
Stable weighted $n$-marked graphs of genus $g$ naturally form a category denoted by $J_{g,n}.$ Objects of this category are stable        weighted $n$-marked graphs and morphisms are given by compositions of contractions of edges and isomorphisms which preserve markings and genus labelings (for details see \cite{KG} \cite{CGP1} and \cite{CGP2}).\footnote{Since there are finitely many isomorphism classes of objects in           $J_{g,n}$ we pick up one object in each equivalence class and pass to the full subcategory.} For a stable $n$-marked graph $\Gamma$ we will denote by $\mathrm{Aut}(\Gamma):=\mathrm {Isom}_{J_{g,n}}(\Gamma,\Gamma)$ a group of automorphisms of $\Gamma.$ For a stable $n$-marked graph         $\Gamma$ and $e\in E(\Gamma)$ we will denote by $\Gamma_e$ a graph that is obtained from $\Gamma$ by contracting an edge $e.$ The category $J_{g,n}$    has the terminal object denoted by $\bullet_{g,n}.$ By $J_{g,n}^k$ we will denote a set of stable $n$-marked graphs $\Gamma$ with $E(\Gamma)=k.$

\par\medskip

\begin{Def}\label{trop1} Fix $g, n \geq0$ with a condition $2g-2+n>0.$ We define a $J_{g,n}$-diagram $\mathcal {M}^{trop}_{g,n}$ called a \textit{moduli diagram of tropical curves of genus $g$ with $n$-punctures}:
$$
\mathcal M^{trop}_{g,n}\colon J_{g,n}^{\circ}\longrightarrow \textsf {Top}
$$
By the following rule:
\par\medskip
For every object $(\Gamma,w,m)\in J_{g,n}$ we set:
$$
\mathcal M^{trop}_{g,n}\colon (\Gamma,w,m)\longmapsto \mathbb R_{\geq 0}^{E(\Gamma)}:=\sigma_{\Gamma}.
$$
For a morphism $f\colon (\Gamma,w,m)\longrightarrow (\Gamma',w',m')$ we set:
$$
\mathcal M^{trop}_{g,n}\colon f\longmapsto \sigma f,
$$
where $\sigma f\colon \sigma_{\Gamma'}\longrightarrow \sigma_{\Gamma}$ is a map that sends an $e'$-coordinate of $\mathbb R_{\geq 0}^{E(\Gamma')}$ to an $e$-coordinate of $\mathbb R_{\geq 0}^{E(\Gamma)}$ if $f$ sends the edge $e$ to the edge $e'$ and zero otherwise.
\end{Def}

\begin{remark} Note that the colimit of the $J_{g,n}^{\circ}$-diagram above:
$$
M_{g,n}^{trop}:=\clim_{J_{g,n}^{\circ}}\mathcal M_{g,n}^{trop}.
$$
is usually called a moduli space of tropical curves of genus $g$ with $n$ punctures \cite{CGP2} \cite{CGP1}.
\end{remark}
Each space $\sigma_{\Gamma}$ carries a natural stratification by strata $\overline{S}_{\Gamma'}:=\mathbb R^{E(\Gamma)}_{\geq 0},$ where $\Gamma'$ is a graph that is defined by contracting some edges of a graph $\Gamma$  (such stratum is of dimension $|E(\Gamma)|.$) Moreover:
$$
S_{\Gamma'}\subset \overline{S}_{\Gamma''} \Leftrightarrow \Gamma'' \rightarrow \Gamma'.
$$
We will denote this stratification by $\mathcal S_{\Gamma}$ and the resulting stratification on $\mathcal M_{g,n}^{trop}$ will be denoted by $\mathcal S.$
\subsection{Sheaves on $\mathcal M_{g,n}^{trop}$}

In this subsection, we define and study combinatorial DG-sheaves $\EuScript {DM}_{g,n}$ on $J_{g,n}^{\circ}$-diagram $\mathcal                    M_{g,n}^{trop}.$ To do it let us recollect some well-known facts about moduli stacks of curves:
\par\medskip
Let $I$ be a finite set. We denote by $\overline{\mathcal M}_{g,I}$ the \textit{moduli stack of stable genus $g$ curves with $|I|$-marked points labeled by the set $I.$} When $2g+|I|-2>0$ this is the smooth and proper Deligne-Mumford stack with the corresponding coarse moduli space denoted by $\overline{M}_{g,I}$ \cite{DM} \cite{KN}. In the special case when $I=\{1,\dots ,n\}$ we use a notation $\overline{\mathcal M}_{g,n}.$ We shall write $\mathcal M_{g,I}$ (resp.  $M_{g,I}$) to indicate the open substack (resp. open coarse moduli space) parameterising smooth curves. We will also use a special notation $\overline{\mathcal M}_g$ for a moduli stack of curves marked by the empty set. 
\par\medskip
For a weighted graph $\Gamma\in J_{g,n}$ we can associate the following stack:
\begin{equation}\label{deco}
\overline{\mathcal M}_{\Gamma}:=\prod_{v\in V(\Gamma)} \overline{\mathcal M}_{w(v),n_v}
\end{equation}
By $\mathcal M_{\Gamma}$ we will denote a product of the corresponding open substacks.
Following \cite{KN} for a morphism $\Gamma\rightarrow \Gamma'$  in the category $J_{g,n}$ we associate \textit{clutching morphisms}:
\begin{equation}\label{cl}
\xi_{\Gamma'\Gamma}\colon \overline{\mathcal M}_{\Gamma}\longrightarrow \overline{\mathcal M}_{\Gamma'},
\end{equation}
which are defined by gluing stable curves at marked points.
An important property of these maps is that they satisfy associativity conditions:
\begin{equation}\label{ass}
\xi_{\Gamma''\Gamma'}\circ \xi_{\Gamma'\Gamma}=\xi_{\Gamma''\Gamma},\qquad \Gamma''\rightarrow \Gamma'\rightarrow \Gamma.
\end{equation}
Denote by $\partial\overline{\mathcal M}_{g,I}:=\overline{\mathcal M}_{g,I} \setminus \mathcal M_{g,I}$ the complement to the substack  parameterising smooth $I$-marked curves. It is known that $\partial\overline{\mathcal M}_{g,I}$ is a normal crossing divisor. We have the          following diagram of stacks:
\begin{equation}\label{Diag1}
\begin{diagram}[height=3em,width=3em]
\mathcal M_{g,I} &   \rTo^{j_I}   & \overline{\mathcal M}_{g,I} &   \lTo^{i_I}   & \partial\overline{\mathcal M}_{g,I}
\end{diagram}
\end{equation}
where a morphism $j$ is open and a morphism $i$ is closed.
\par\medskip
For any $\Gamma\in J_{g,n}$ denote by $\mathcal D_{\Gamma}$ a substack in the moduli stack $\overline{\mathcal M}_{g,I},$ defined as a closure of the locus of stable curves with the dual graph given by $\Gamma.$ The boundary divisor admits the following decomposition:
$$
\partial\overline{\mathcal M}_{g,I}=\coprod_{\Gamma\in J_{g,n}} \mathcal D_{\Gamma}
$$
Denote by $\widetilde{\mathcal D}_{\Gamma}\longrightarrow \mathcal D_{\Gamma}$ the corresponding normalised stack. Note that  $\widetilde{D}_{\Gamma}$ is the smooth stack. We have the following equivalence of stacks:
\begin{equation}\label{norm}
\overline{\mathcal M}_{\Gamma}/ \mathrm {Aut}(\Gamma)\cong \widetilde{\mathcal D}_{\Gamma},
\end{equation}
Let $\Gamma$ and $\Gamma'$ be objects of the category $J_{g,n}$ such that $\Gamma$ has exactly $k$ edges such that if one contracts one of them the resulting graph is isomorphic to $\Gamma'.$ Denote by $F$ a set of these edges. Then a group $\mathrm {Aut}(\Gamma)$ acts on $\overline{\mathcal M}_\Gamma\times F$ and we set $\widetilde{\mathcal D}_{\Gamma',\Gamma}:=\overline{\mathcal M}_\Gamma\times F / \mathrm {Aut}(\Gamma).$ We have the following correspondence:
\begin{equation}\label{norm2}
\begin{diagram}[height=2.5em,width=2.5em]
\widetilde{\mathcal D}_{\Gamma',\Gamma} &   \rTo^{\widetilde{\xi}_{\Gamma',\Gamma}}   &   \widetilde{\mathcal D}_{\Gamma'}  &  \\
\dTo^{\pi_{\Gamma',\Gamma}} & &  && \\
\widetilde{\mathcal D}_{\Gamma} &      &    \\
\end{diagram}
\end{equation}
Where $\pi_{\Gamma',\Gamma}$ is a standard projection map, and $\widetilde{\xi}_{\Gamma',\Gamma}$ is defined by the following rule: with an element     in $\overline{\mathcal M}_\Gamma$ and an edge $e\in F,$ the clutching along $e$ produces an element that is well defined up to an automorphism of $\Gamma'.$ This construction gives us a morphism $\overline{\mathcal M}_\Gamma\times F\longrightarrow \widetilde{\mathcal D}_{\Gamma'}$ which factors through               $\widetilde{\mathcal D}_{\Gamma',\Gamma}.$

\par\medskip
For a stratified space $\sigma_{\Gamma}$ associated with a stable weighted $n$-marked graph $\Gamma\in J_{g,n}$ we define an $\mathcal S_{\Gamma}$-smooth DG-combinatorial sheaf    $\EuScript {DM}_{\Gamma}\in \textsf {Sh}_c(\sigma_{\Gamma},\mathcal S)$ by the rule:
\begin{itemize}
  \item \par\medskip
For every graph $\Gamma'\in J_{g,n,\Gamma}^{\circ}$ we set:
$$
R\Gamma(S_{\Gamma'},\EuScript {DM}_{\Gamma}):=\bigotimes_{v\in V(\Gamma')} C^{\hdot}(\overline{\mathcal M}_{w(v),n_v},\mathbb Q)
$$
\par\medskip
  \item For every inclusion of strata $S_{\Gamma'} \subset \overline{S}_{\Gamma''}$ we define a variation map:
$$
var_{\Gamma', \Gamma''}\colon R\Gamma(S_{\Gamma'},\EuScript{DM}_{\Gamma})\longrightarrow R\Gamma(S_{\Gamma''},\EuScript{DM}_{\Gamma}),\quad var_{\Gamma', \Gamma''}:=\xi_{\Gamma',\Gamma''}^*
$$
as being induced by a pullback along the clutching morphism.
\end{itemize}
Since clutching maps satisfy the associativity condition we have the well-defined combinatorial DG-sheaf on $\sigma_{\Gamma}.$
\begin{Def} For every $g, n \geq0$ with a condition $2g-2+n>0$ we define an $\mathcal S$-smooth combinatorial DG-sheaf                                                     $\EuScript{DM}_{g,n}:=\{\EuScript{DM}_{\Gamma}\}_{\Gamma\in J_{g,n}}$ on the $J_{g,n}^{\circ}$-diagram $\mathcal M^{trop}_{g,n},$ called a \textit{Deligne-Mumford DG-sheaf} with connecting quasi-isomorphisms:
$$
\alpha(m)\colon m^*\EuScript{DM}_{\Gamma}\overset{\sim}{\longrightarrow}\EuScript{DM}_{\Gamma'},\quad m\colon \Gamma\longrightarrow \Gamma'
$$
Where $\alpha(m)$ is a natural quasi-isomorphism if $m$ is an edge contraction and $\alpha(m)$ is a permutation morphism if $m$ is an automorphism.     It is easy to see that all properties from Definition \ref{DGSh} are satisfied.
\end{Def}

\par\medskip

Recall that for any complex $C$ in an abelian category a standard truncation functors $\tau_{\leq k}$ define an increasing filtration:
$$
\dots \subset \tau_{\leq k}C\subset \tau_{\leq k+1}C\subset \dots\subset C.
$$
By applying these truncation functors to the definition of Deligne-Mumford DG-sheaves and using \eqref{comp5} we get a sequence of DG-combinatorial sheaves on $\mathcal M_{g,n}^{trop}:$
$$
\EuScript W_0\EuScript{DM}_{g,n}\rightarrow \EuScript W_1\EuScript{DM}_{g,n}\rightarrow\dots \rightarrow \EuScript{DM}_{g,n}
$$
\begin{Def} For every $k\geq0$ and $g,n$ such that $2g-2+n>0$ we define an $\mathcal S$-smooth DG-combinatorial sheaf $\mathrm {Gr}^{\EuScript W}_k\EuScript{DM}_{g,n}$ on the $J_{g,n}^{\circ}$-diagram $\mathcal M_{g,n}^{trop}$ called a \textit{Deligne-Mumford sheaf of the weight $k$} by the rule:
$$
\mathrm {Gr}^{\EuScript W}_k\EuScript{DM}_{g,n}:=\mathrm {Cone}(\EuScript W_{k-1}\EuScript{DM}_{g,n}\rightarrow \EuScript                             W_{k}\EuScript{DM}_{g,n})
$$

\end{Def}
Recall that for a complex $C$ in an abelian category we have a natural projection morphism $\tau_{\geq k}C/           \tau_{\geq k-1}C\longrightarrow H^k(C)[k]$ which induces the isomorphism on cohomology. Hence for every $k,$ we have the quasi-isomorphism of          combinatorial DG-sheaves:
\begin{equation}\label{form}
\mathrm {Gr}^{\EuScript W}_k\EuScript{DM}_{g,n}\overset{\sim}{\longrightarrow}H^k(\EuScript{DM}_{g,n})[k].
\end{equation}

\begin{remark}\label{modop1} Denote by $\textsf{ModOp}_{\mathcal D}$ a category of $\mathcal D$-twisted DG-modular cooperads in the sense E. Getzler and M. Kapranov \cite{KG}. Then using the description of combinatorial sheaves on $\mathcal M_{g,n}^{trop}$ one can show that equivalence \eqref{comp5} defines the functor:
$$
\textsf {Ho}(\textsf{ModOp}_{\mathcal D})\longrightarrow\textsf {D}_c(\mathcal M_{g,n}^{trop},\mathcal S)
$$
Under this functor, the combinatorial DG-sheaf $\EuScript{DM}_{g,n}$ corresponds to the (non-twisted) DG-modular cooperad               $C^{\hdot}(\overline{\mathcal M},\mathbb Q)$ from \cite{KG}.
\end{remark}

\subsection{Getzler-Kapranov complexes.} In this subsection, we compute the cohomology with compact support of combinatorial DG-sheaves
$\EuScript{DM}_{g,n}.$ To do it we give the following:
\begin{Def}\label{GKK} The \textit{Getzler-Kapranov complex} $\textsf {GK}_{g,n}$ is defined as the decorated graph complex which computes the compactly supported cohomology
with coefficients in the Deligne-Mumford sheaf:                                                                                                       $$
\textsf{GK}_{g,n}:=R\Gamma_c(\mathcal M^{trop}_{g,n},\EuScript{DM}_{g,n})
$$

\end{Def}

Following \cite{KG} we compute the cohomology of the Getzler Kapranov complex:

\begin{Prop}\label{KD} We have the following quasi-isomorphism of DG-vector spaces:
$$
C_c^{\hdot}(\mathcal M_{g,n},\mathbb Q)\overset{\sim}{\longrightarrow}\textsf{GK}_{g,n}.
$$
\end{Prop}

\begin{proof}

For any $\Gamma\in J_{g,n}$ by Proposition \ref{GK} a complex which computes the cohomology of $R\Gamma_c(\sigma_{\Gamma}, \EuScript {DM}_{\Gamma})$
can be represented as the total complex of the following DG-vector space:                                                                                 \begin{align}
C^{\hdot}(\overline{\mathcal M}_{g,n},\mathbb Q)\rightarrow \dots \rightarrow  \bigoplus_{\Gamma'\in J^k_{g,n}\,,\Gamma'\in J_{g,n,\Gamma}^{\circ}}\bigotimes_{v\in V(\Gamma')} C^{\hdot}(\overline{\mathcal M}_{w(v),n_v},\mathbb Q)\otimes\mathrm {det}(\Gamma')\rightarrow\dots
\end{align}
Here $\det(\Gamma'):=\det(E(\Gamma'))\cong \textsf {or}_{S_{\Gamma}}$ is the determinant on the set of edges of $\Gamma'$ and graphs with $k$ edges are placed in degree $k.$ A differential $D$ is defined by the rule:                                                                                                                         $$
D=\sum_{\Gamma\in J^k_{g,n},\,\Gamma'\in J_{g,n}^{k-1}\,\Gamma\twoheadrightarrow \Gamma'} \xi^*_{\Gamma',\Gamma}.
$$
To compute the homotopy limit over $J_{g,n}^{\circ}$ we do the following trick:
\par\medskip
The canonical functor to the terminal category $a_{J_{g,n}^{\circ}}\colon J^{\circ}_{g,n}\longrightarrow \mathbf 1$ can be decomposed
as $J_{g,n}^{\circ}\longrightarrow \mathbf N \longrightarrow \mathbf 1,$ where $\mathbf N$ is a category associated with the poset of natural
numbers and the first functor sends every graph $\Gamma$ with $k$ edges to the natural
number $k.$ Hence:
$$
\lim_{J^{\circ}_{g,n}}\cong \lim_{\mathbf N}\circ \mathrm {LKan}_P,
$$
where $\mathrm {LKan}_P$ is a left Kan extension functor along the functor $P.$ Thus since every graph $\Gamma$
has a finite automorphism group $\mathrm {Aut}(\Gamma)$ and we work over $\mathbb Q$ it is easy to see that the left Kan extension functor $\mathrm {LKan}_P$ along the functor $P$ is the exact functor.
Thus applying this construction to the DG-combinatorial                                                                                               sheaf $\EuScript {DM}_{g,n}$ we get the functor from $\mathbf N,$ with a value at the number $n$
is given by the total complex of the following DG-vector space:
\begin{align}
C^{\hdot}(\overline{\mathcal M}_{g,n},\mathbb Q)\rightarrow \dots \rightarrow  \bigoplus_{\Gamma\in J^k_{g,n}}\left(\bigotimes_{v\in V(\Gamma)}
C^{\hdot}(\overline{\mathcal M}_{w(v),n_v},\mathbb Q)\otimes\mathrm {det}(\Gamma)\right)^{\mathrm {Aut}(\Gamma)}\rightarrow\dots                             \end{align}
Where in the last non-zero entry, we place graphs $\Gamma$ with $E(\Gamma)=n.$ A differential is defined by the following rule:
\begin{equation}\label{diffe}
\widetilde{D}=\sum_{\Gamma\in J^k_{g,n}\,,\Gamma'\in J_{g,n}^{k-1}\,\Gamma\twoheadrightarrow \Gamma'}
\pi_{\Gamma',\Gamma!}\circ\widetilde{\xi}_{\Gamma',\Gamma}^*.                                                                                         \end{equation}
This functor satisfies a generalized Mittag-Leffler condition and hence the limit of this functor
over $\mathbf N$ is exact. Hence the homotopy limit $\hlim_{J_{g,n}^{\circ}}R\Gamma_c(\sigma_{\Gamma},\EuScript {DM}_{\Gamma})$ is exact  and therefore                                                                                                                                      $R\Gamma_c(\mathcal M_{g,n}^{trop},\EuScript {DM}_{g,n})$ is quasi-isomorphic to the total DG-vector space of the following
complex:
\begin{align}
C^{\hdot}(\overline{\mathcal M}_{g,n},\mathbb Q)\rightarrow \dots \rightarrow  \bigoplus_{\Gamma\in J^k_{g,n}}\left(\bigotimes_{v\in V(\Gamma)}
C^{\hdot}(\overline{\mathcal M}_{w(v),n_v},\mathbb Q)\otimes\mathrm {det}(\Gamma)\right)^{\mathrm {Aut}(\Gamma)}\rightarrow\dots                             \end{align}
By a Künneth formula, the functor of the compactly supported cochains is monoidal and also it commutes with homotopy colimits. Hence we get the following quasi-isomorphism:                                                                                                                                    \begin{align}
  \bigoplus_{\Gamma\in J_{g,n}^k}\left(\bigotimes_{v\in V(\Gamma)} C^{\hdot}(\overline{\mathcal M}_{w(v),n_v},\mathbb Q)\otimes \mathrm
  {det}(\Gamma)\right )^{\mathrm {Aut}(\Gamma)}\cong\\                                                                                                  C^{\hdot}\left(\coprod_{\Gamma\in J_{g,n}^k} \prod_{v\in V(\Gamma)}\overline{\mathcal M}_{w(v),n_v}/ \mathrm {Aut}(\Gamma),\varepsilon_k\right).
\end{align}
Where $\varepsilon$ is a local system defined by a sign representation of a group that permutes the edges of a stable graph. By \eqref{norm} the latter stack is equivalent to the stack $\widetilde{\mathcal D}_{k},$ which is a normalisation of the stack of nodal curves with at least $k$ nodes. A differential is given by the pull-push along diagram \eqref{norm2}. Therefore by Proposition \ref{D1} we have the      desired result.

\end{proof}

\begin{remark}\label{FT} Note that using a notion of a constructible DG-cosheaf on the $J_{g,n}^{\circ}$-diagram $\mathcal M_{g,n}^{trop}$ similarly to Remark
\ref{modop1} one can encode a notion of a $\mathcal D$-twisted modular operad:                                                                        
\begin{equation}\label{modreal1}
\textsf{ModOp}_{\mathcal D}\longrightarrow \textsf {D}_{c}^{ch}(\mathcal M_{g,n}^{trop},\mathcal S).
\end{equation} 
Moreover one can show that the following diagram is commutative:
\begin{equation}
\begin{diagram}[height=2.5em,width=5em]
 \textsf {Ho}(\textsf{ModCoop}_{\mathcal D}) &   \rTo^{\mathcal F_{\mathcal D}}   &   \textsf {Ho}(\textsf{ModOp}_{\mathcal D^{\vee}})  &  \\
\dTo_{}^{} & &  \dTo_ {}^{}  && \\
\textsf {D}_{c}(\mathcal M_{g,n}^{trop},\mathcal S) &   \rTo_{}^{\mathbb V}   & \textsf {D}_{c}^{ch}(\mathcal M_{g,n}^{trop},\mathcal S)
\end{diagram}
\end{equation}
Where $\mathcal F_{\mathcal D}$ is a covariant $\mathcal D$-Feynman transform \cite{KG} and $\mathbb V$ is a covariant
Verdier duality functor \cite{LG}. Under functor \eqref{modreal1} a DG-cosheaf $\mathbb V(\EuScript            {DM}_{g,n})$ corresponds to the twisted modular operad $\mathcal{G}rav.$ By Poincaré-Verdier duality we have:\footnote{By $H^{\hdot}_+$ we understand a (cosheaf) pushforward functor \cite{LG} along the morphism to a point.}                                                                $$                                                                                                                                                    H^{\hdot}_+(\mathcal M_{g,n}^{trop},\mathbb V(\EuScript {DM}_{g,n}))\cong H^{\hdot}_c(\mathcal M_{g,n}^{trop},\EuScript {DM}_{g,n})
$$
All these facts together with a formality of the DG-modular cooperad $C^{\hdot}(\overline{\mathcal M},\mathbb Q)$ \cite{form} imply an equivalence between Definition \ref{GKK} and the definition from \cite{AWZ}.
\end{remark}

\subsection{The weight filtration on $H_c^{\hdot}(\mathcal M_{g,n},\mathbb Q)$} Note that since $\mathcal M_{g,n}$ is the smooth Deligne-Mumford stack for every
$k$ the corresponding cohomology with compact support $H_c^k(\mathcal M_{g,n},\mathbb Q)$ has weights in a region $\{0,\dots ,k\}$ i.e. \textit{a priori} only the corresponding graded quotients are non zero. From Proposition \ref{D2} we obtain a description of the weight graded quotients of $H_c^{\hdot}(\mathcal M_{g,n},\mathbb Q)$ in terms of  decorated graph complexes.
\begin{Def} For $2g-2+n>0$ we denote by $\textsf W_k\textsf{GK}_{g,n}$ a decorated graph complex which computes the cohomology with compact support of the $J_{g,n}^{\circ}$-diagram $\mathcal M_{g,n}^{trop}$ with coefficients in the DG-combinatorial sheaf $\mathrm {Gr}^{\EuScript
W}_k\EuScript{DM}_{g,n}:$                                                                                                                             $$                                                                                                                                                    R\Gamma_c(\mathcal M^{trop}_{g,n},\mathrm {Gr}^{\EuScript W}_k\EuScript{DM}_{g,n}):=\textsf W_k\textsf{GK}_{g,n}
$$
This DG-vector space will be called the \textit{Getzler-Kapranov complex of the weight $k$.}
\end{Def}

Our main result in this section is:

\begin{Th}\label{wgk} We have the following description of the weight filtration on the compactly supported cohomology of $\mathcal M_{g,n}:$

$$
\mathrm {Gr}^W_k H^{i+k}_c(\mathcal M_{g,n})\overset{\sim}{\longrightarrow}  H^i(\textsf W_k\textsf{GK}_{g,n})
$$

\end{Th}
\begin{proof} Let us compute the complex:
\begin{equation}\label{gk1}
R\Gamma_c(\mathcal M^{trop}_{g,n},\mathrm {Gr}_k^{\EuScript W}\EuScript{DM}_{g,n})
\end{equation}
Applying the quasi-isomorphism $\mathrm {Gr}^{\EuScript W}_k\EuScript{DM}_{g,n}\overset{\sim}{\longrightarrow}H^k(\EuScript{DM}_{g,n})[k]$ of
combinatorial DG-sheaves on $\mathcal M_{g,n}^{trop}$ and acting like in the proof of Proposition \ref{KD} we obtain that \eqref{gk1} is quasi-isomorphic to the total DG-vector space of the complex:                                                                  \begin{equation}                                                                                                                                      \begin{diagram}[height=3em,width=3em]
H^k(\overline{\mathcal M}_{g,n},\mathbb Q)&   \rTo^{\widetilde{D}}   & H^k(\widetilde{\mathcal D}_1,\varepsilon_1)&   \rTo^{\widetilde{D}}   & \dots &   \rTo^{\widetilde{D}}   &
H^k(\widetilde{\mathcal D}_p,\varepsilon_p)\rightarrow\dots,
\end{diagram}                                                                                                                                         \end{equation}
where $\widetilde{D}$ is defined by \eqref{diffe}. By Proposition \ref{D2} the $i$-cohomology of the complex above coincides with $\mathrm {Gr}^W_k
H^{i+k}_c(\mathcal M_{g,n},\mathbb Q).$

\end{proof}

\subsection{Low weights} Denote by $\textsf {H}_{n}\textsf {GC}^{\hdot}(\delta)$, the \textit{$n$-labelled hairy graph complex} in
the sense \cite{AWZ} and \cite{CGP2}. Elements of this complex are stable weighted $n$-marked graphs without loops and with at least
trivalent vertices. The differential $\delta$ is defined as the signed sum of vertex splitting operators. In the special case when $n$
is zero this complex coincides with $\textsf {GC}^{\hdot}(\delta)$ the \textit{M. Kontsevich graph complex} \cite{Will}. A cohomological grading on     $\textsf {H}_n\textsf {GC}^{\hdot}(\delta)$ is defined by a number of edges. Since the differential in $\textsf {H}_n\textsf {GC}^{\hdot}(\delta)$ does not change the genus of a graph, complex $\textsf {H}_n\textsf {GC}^{\hdot}(\delta)$ can be decomposed as:
$$
\textsf {H}_n\textsf {GC}^{\hdot}(\delta):=\prod_{g=0}^{\infty} \textsf{B}_g\textsf {H}_n\textsf {GC}^{\hdot}(\delta),
$$
where $\textsf{B}_g\textsf {H}_n\textsf {GC}_0^{\hdot}(\delta),$ is the subcomplex that consists of stable $n$-marked graphs of genus $g.$

\begin{Cor}\label{CGP11} We have the following quasi-isomorphism of complexes:
$$
\textsf{B}_g\textsf {H}_n\textsf {GC}^{\hdot}(\delta)\cong \textsf W_0\textsf{GK}_{g,n}^{\hdot}
$$
\end{Cor}

\begin{proof}  Consider the weight zero combinatorial DG-sheaf $\EuScript W_0\EuScript {DM}_{g,n}.$ From the irreducibility of the moduli stack of
curves of genus $g$ with $n$-punctures \cite{DM} we have an isomorphism of sheaves:
$$                                                                                                                                                    \EuScript W_0\EuScript {DM}_{g,n} \cong  \underline{\mathbb Q}_{\mathcal M_{g,n}^{trop}}
$$
Note that in our case \eqref{comp4} is the quasi-isomorphism (an automorphism group of a stable graph acts freely). Hence we have: 
$$
R\Gamma_c(\mathcal M_{g,n}^{trop},\underline{\mathbb Q}_{\mathcal M_{g,n}^{trop}}):=\hlim_{J^{\circ}_{g,n}}
R\Gamma_c(\sigma_{\Gamma},\underline{\mathbb Q}_{\Gamma})\cong R\Gamma_c(M_{g,n}^{trop},\mathbb Q).
$$                                                                                                                                                    Recall (for details see \cite{CGP1} and \cite{CGP2}) that a moduli space of volume $1$ tropical curves of genus $g$ with $n$ punctures $\Delta_{g,n}$ is a certain symmetric $\Delta$-complex with the geometric realisation $|\Delta_{g,n}|$
identified with the complement to a link in $M_{g,n}^{trop},$ hence:
$$
H^k_c(M_{g,n}^{trop},\mathbb Q)\cong \widetilde{H}^{k-1}(|\Delta_{g,n}|,\mathbb Q).
$$
Further using \textit{ibid.} one may decompose:
$$
C^{\hdot}(|\Delta_{g,n}|,\mathbb Q)=C^{\hdot}(|\Delta_{g,n}^{lw}|,\mathbb Q)\oplus C^{\hdot}(|\Delta_{g,n}|,|\Delta_{g,n}^{lw}|,\mathbb Q).
$$
Where $\Delta_{g,n}^{lw}$ is a sub symmetric $\Delta$-complex which consists of graphs with loops and vertices of positive weights.
Applying Proposition \cite{CGP2} it can be shown that $\Delta_{g,n}^{lw}$ is contractible therefore by the definition of the hairy $n$-marked complex we
get:
$$                                                                                                                                                    C^{\hdot}(|\Delta_{g,n}|,|\Delta_{g,n}^{lw}|,\mathbb Q):=\textsf B_g\textsf {H}_n\textsf {GC}^{\hdot}(\delta).
$$
Hence we get the desired quasi-isomorphism.

\end{proof}

\begin{remark}  For $g,n$ with $2g-2+n>0$ consider the canonical morphism of combinatorial DG-sheaves
\begin{equation}\label{awz}
\EuScript W_0\EuScript {DM}_{g,n}\longrightarrow \EuScript {DM}_{g,n}
\end{equation}
Applying the functor of the compactly supported cohomology we get the morphism:
\begin{equation}\label{GCP1}
\textsf{cgp}\colon H^{\hdot}(\textsf B_g\textsf {H}_n\textsf {GC})\longrightarrow H^{\hdot}_c(\mathcal M_{g,n},\mathbb Q)
\end{equation}
From the result above we see that this morphism is injective and tautologically coincides with a morphism constructed in \cite{CGP1} and \cite{CGP2}. We can also describe \eqref{GCP1}
from a little bit different perspective. The correspondence from Remark \ref{FT} and the formality result \cite{form} applied to \eqref{awz} produce the morphism of graded modular
cooperads \cite{AWZ}:
$$
\mathcal {C}omm\longrightarrow C^{\hdot}(\overline{\EuScript M},\mathbb Q),
$$
where $\mathcal {C}omm$ is a modular envelope of the cyclic cocommutative
cooperad. Note that the compactly supported cochains with coefficients in $\EuScript W_0\EuScript {DM}_{g,n}$ and $\EuScript {DM}_{g,n}$ can be
identified with the DG-vector spaces over $\bullet_{g,n},$ of the Feynman
transforms of the corresponding modular cooperads. Hence the morphism $\textsf {cgp}$ coincides with the morphism from \cite{AWZ}.
\end{remark}

 Recall that in \cite{AC} (Theorem $2.1$) it was proved that $H^k(\overline{\mathcal M}_{g,n},\mathbb Q)=0$ for $k=1,3,5$ and all $g,n$ such that $2g-2+n>0.$ Thus by quasi-isomorphism \eqref{form} and by the Künneth formula, we have that the cohomology of the Deligne-Mumford sheaves $\mathrm {Gr}^{\EuScript W}_k\EuScript {DM}_{g,n}$  vanishes when $k=1,3,5.$ We have the following:

\begin{Cor} We have the vanishing of the $k$-associated weight quotients of $H^{i+k}_c(\mathcal M_{g,n},\mathbb Q)$ in the case when $k=1,3,5:$
$$
\mathrm {Gr}_k^WH^{i+k}_c(\mathcal M_{g,n},\mathbb Q)=0,\qquad k=1,3,5.
$$
\end{Cor}

\begin{remark} $(i)$ The morphism $\textsf{cgp}$ allows producing previously unknown classes in moduli stacks of curves. In particular when $n=0$ by Theorems of \cite{Will} and \cite{Brown} one gets an injection:
$$
\widehat{Free}_{Lie}(\sigma_3,\dots,\sigma_{2k+1},\dots)\hookrightarrow \prod_{g=3}^{\infty} H_c^{2g}(\mathcal M_g,\mathbb Q).
$$
Where $\widehat{Free}_{Lie}(\sigma_3,\dots,\sigma_{2k+1},\dots)$ is a (completed) free Lie algebra on odd generators, see a discussion in \cite{CGP2}.

$(ii)$ In the recent interesting work \cite{PW1} the cohomology of the weight two Getzler-Kapranov complex was computed. It would be very interesting to calculate the cohomology of $\textsf {W}_4\textsf {GK}_{g,n}^{\hdot}.$

\end{remark}

\section{Sheaves on moduli spaces of hairy tropical curves}

\subsection{Moduli spaces of hairy tropical curves}

For every $g\geq 2$ we consider a category $HJ_{g}$ with objects defined by stable weighted $n$-marked graphs $\Gamma$ of genus $g,$ where the marking function:
$$m\colon \{1,\dots, n\}\longrightarrow V(\Gamma)$$
is considered to be defined for an \textit{arbitrary} value of $n$ i.e. $n\in \mathbb N.$ Morphisms in this category are defined as compositions of contractions of edges, morphisms that forget markings, and isomorphisms of graphs that may not preserve labelings of markings. Similar to the case of the category $J_{g,n}$ it is convenient to represent an object in the category $HJ_{g}$ as a stable graph of genus $g$ with any finite number of unlabelled hairs attached to vertices (such that the stability condition holds for each vertex). For $\Gamma\in HJ_g$ we will denote by $\mathrm {Aut}^{h}(\Gamma):=\mathrm {Isom}_{HJ_{g}}(\Gamma,\Gamma)$ a group of automorphisms of a graph $\Gamma$ in the category $HJ_g.$ It is important to distinguish it from the group of automorphisms of $\Gamma$ considered as an object of $J_{g,n}$ (for example $\mathrm {Aut}^{h}(\bullet_{g,n})\cong \Sigma_n,$ while $\mathrm {Aut}(\bullet_{g,n})$ is trivial). For a graph $\Gamma\in HJ_{g}$ and $h\in m^{-1}(V(\Gamma))$ we denote by $\Gamma_h$ an element in $HJ_{g}$ with a hair $h$ being contracted.\footnote{We consider this operation when the resulting graph $\Gamma_h$ is stable.} By $HJ_g^{n,k}$ we denote a collection of graphs in $HJ_g$ with exactly $n$-edges and $k$-hairs. Analogous to the definition of the $J_{g,n}^{\circ}$-diagram $\mathcal M_{g,n}^{trop}$ we give:

\begin{Def}\label{hartrop} For every $g\geq 2$ we define a $HJ^{\circ}_g$-diagram $\mathcal  {HM}_{g}^{trop}$ called a \textit{moduli diagram of tropical curves of genus $g$ with an arbitrary number of hairs}:
$$
\mathcal {HM}_{g}^{trop}\colon HJ_{g}^{\circ} \longrightarrow \textsf {Top}
$$
by the following rule:
\par\medskip
For every object $(\Gamma,w,m)\in HJ_{g}$ we set:
$$
\mathcal {HM}_{g}^{trop}\colon (\Gamma,w,m)\longmapsto \sigma_{\Gamma}^{hair}:=\mathbb R_{\geq 0}^{E(\Gamma)\sqcup m^{-1}(V(\Gamma))}
$$
For a morphism $f\colon (\Gamma,w,m)\longrightarrow (\Gamma',w',m')$ in $HJ_{g}$ we set:
$$
\mathcal {HM}_{g}^{trop}\colon \colon f\longmapsto \sigma^{hair} f,
$$
where $\sigma^{hair} f\colon \sigma_{\Gamma'}^{hair}\longrightarrow \sigma_{\Gamma}^{hair}$ is a map that sends an $e'$-coordinate (resp. an $h'$-coordinate) of a topological space $\mathbb R_{\geq 0}^{E(\Gamma')\sqcup m^{-1}(V(\Gamma'))}$ to an $e$-coordinate (resp. an $h$-coordinate) of space $\mathbb R_{\geq 0}^{E(\Gamma)\sqcup m^{-1}(V(\Gamma'))}$ if $f$ sends the edge $e$ (resp. the hair $h$) to the edge $e'$ (resp. the hair $h'$) and zero otherwise.
\end{Def}
Analogous to the case of moduli diagrams $\mathcal M_{g,n}^{trop}$ each topological space $\sigma^{hair}_{\Gamma}$ is naturally stratified by graphs $\Gamma'\in HJ_{g,\Gamma}^{\circ}.$ We will denote this stratification by $\mathcal S_{\Gamma}^{hair}$ and the resulting stratification on the $HJ_g^{\circ}$-diagram $\mathcal {HM}_{g}^{trop}$ by $\mathcal S^{hair}.$
\par\medskip
We also consider a natural $HJ_{g}^{\circ}$-diagram $\mathcal{HM}_{g,\emptyset}^{trop}$ defined by sending each stable weighted marked graph $\Gamma$ to the topological space $\sigma_{\Gamma}.$ We have a canonical closed morphism between diagrams $v_{\emptyset}\colon \mathcal{HM}_{g,\emptyset}^{trop}\hookrightarrow \mathcal{HM}^{trop}_{g}.$ The complement $HJ_{g}^{\circ}$-diagram will be denoted by $\mathcal{HM}_{g,\geq 1}^{trop}.$ We will denote a value of this diagram at a stable graph $\Gamma$ by $h_{\geq 1}\sigma_{\Gamma}^{hair}.$ For every $g\geq 2$ we have a sequence of morphisms of $HJ_{g}^{\circ}$-diagrams:
\begin{equation}\label{Diag2}
\begin{diagram}[height=3em,width=3em]
\mathcal H_{\geq 1}\mathcal M_{g}^{trop} &   \rTo^{v_{\geq 1}}   & \mathcal{HM}_{g}^{trop} &   \lTo^{v_{\emptyset}}   & \mathcal H_{\emptyset}\mathcal M_{g}^{trop}
\end{diagram}
\end{equation}

\subsection{Sheaves on $\mathcal{HM}^{trop}_g$}

We will define an analog of the combinatorial Deligne-Mumford DG-sheaves $\EuScript {DM}_{g,n}$ on the diagram $\mathcal {HM}^{trop}_{g}.$ To do it let us recall that for every finite set $I$ with a subset $J\subset I$ there is a surjective morphism of stacks:
\begin{equation}\label{forg}
\pi_{J}\colon \overline{\mathcal M}_{g,I}\longrightarrow \overline {\mathcal M}_{g,I\setminus J},
\end{equation} 
defined by forgetting $J$-labelled marked points and stabilising the resulting nodal curve \cite{KN}. In particular, we will use a notation $\pi$ when $J=I.$ 
\par\medskip
For a stratified space $\sigma^{hair}_{\Gamma}$ associated with the stable marked graph $\Gamma\in HJ_g$ we define an $\mathcal S_{\Gamma}^{hair}$-smooth combinatorial DG-sheaf $\EuScript {DM}_{\Gamma}^{hair}$ on $\sigma^{hair}_{\Gamma}$ by the following rule:
\begin{itemize}
  \item For every graph $\Gamma'\in HJ_{g,\Gamma}^{\circ}$ we set:
$$
R\Gamma(S_{\Gamma'},\EuScript{DM}_{\Gamma}^{hair}):=\bigotimes_{v\in V(\Gamma')} C^{\hdot}(\overline{\mathcal M}_{w(v),n_v},\mathbb Q)
$$
  \item For every inclusion of strata $S_{\Gamma''} \subset \overline{S}_{\Gamma'}$ we define a variation morphism:
$$
var_{\Gamma'', \Gamma'}\colon R\Gamma(S_{\Gamma'},\EuScript{DM}_{\Gamma}^{hair})\longrightarrow R\Gamma(S_{\Gamma''},\EuScript{DM}_{\Gamma}^{hair})
$$
by the following rule:
\begin{enumerate}[(i)]
\par\medskip
  \item If the inclusion of strata is induced by a contraction of edges we define the variation operator as it was defined in the case of the DG-sheaf $\EuScript {DM}_{g,n}.$
\par\medskip
  \item If $\Gamma''$ is obtained from $\Gamma'$ by a contraction of hairs $h_{v_j}^i$ at a vertex $v_j$ we define $var_{\Gamma'', \Gamma'}$ as the composition of morphisms:
$$
\mathrm {id}\otimes \dots \otimes \pi_{h_{v_j}^i}^*\otimes \dots \otimes \mathrm {id},
$$
where a morphism: $$\pi_{h_{v_j}^i}^*\colon C^{\hdot}(\overline{\mathcal M}_{w(v_j),H_v\sqcup m^{-1}(v_j)},\mathbb Q)\longrightarrow C^{\hdot}(\overline{\mathcal M}_{w(v_j),H_{v_{j}}\sqcup m^{-1}(v_j)\sqcup h_{v_{j}}^i},\mathbb Q)$$ is placed at the vertex $v_j.$
\end{enumerate}

\end{itemize}
The direct check shows that for every $\Gamma\in HJ_g$ the combinatorial DG-sheaf $\EuScript{DM}_{\Gamma}^{hair}$ is well defined.

\begin{Def} For every $g\geq 2$ we define an $\mathcal S^{hair}$-smooth combinatorial DG-sheaf $\EuScript{DM}_{g}^{hair}:=\{\EuScript{DM}_{\Gamma}^{hair}\}_{\Gamma\in HJ_{g,n}}$ on the $HJ_{g,n}^{\circ}$-diagram $\mathcal {HM}^{trop}_{g},$ called a \textit{hairy Deligne-Mumford DG-sheaf} with the connecting quasi-isomorphisms:
$$
\alpha(m)\colon m^*\EuScript{DM}_{\Gamma}^{hair}\overset{\sim}{\longrightarrow}\EuScript{DM}_{\Gamma'}^{hair},\quad m\colon \Gamma\longrightarrow \Gamma'.
$$
Where $\alpha(m)$ is a natural quasi-isomorphism if $m$ is an edge contraction or a hair contraction and $\alpha(m)$ is a permutation morphism if $m$ is an automorphism of a graph. It is easy to see that all properties from Definition \ref{DGSh} are satisfied.
\end{Def}

\subsection{The hairy Getzler-Kapranov complex}
Analogous to Definition \ref{GKK} we give:
\begin{Def}\label{GKKK} For every $g\geq 2$ we define the \textit{hairy Getzler-Kapranov complex} $\textsf H\textsf {GK}_g^{\hdot}$ to be the complex that computes the compactly supported cohomology of the $HJ_g^{\circ}$-diagram $\mathcal {HM}_{g}^{trop}$ with coefficients in the hairy Deligne-Mumford sheaf $\EuScript{DM}_{g}^{hair}:$
$$
\textsf{HGK}_g^{\hdot}:=R\Gamma_c(\mathcal {HM}^{trop}_g,\EuScript{DM}_{g}^{hair})
$$
\end{Def}

\begin{remark} The definition of the hairy Getzler-Kapranov complex was also presented in \cite{AWZ}. In \textit{ibid.}, it is defined as the total complex associated with the skew symmetrisation of the Feynman transform of the modular cooperad $H^{\hdot}(\overline{\mathcal M},\mathbb Q)$ with an additional differential defined by adding a decorated hair. Note that methods of \cite{form} do not directly imply a formality quasi-isomorphism in this case. However, the methods of \cite{CH} do imply the formality quasi-isomorphism and therefore Definition \ref{GKKK} of the hairy Getzler-Kapranov complex coincides with one given in \cite{AWZ}. We thank Dan Petersen for pointing out this reference.

\end{remark}

Our first result (also proved in \cite{AWZ} (Theorem 30)) concerning the cohomology of the hairy Getzler-Kapranov complex is the following:

\begin{Prop}\label{van} The hairy Getzler-Kapranov complex $\textsf{HGK}_g^{\hdot}$ is acyclic.

\end{Prop}
\begin{proof} To prove this result, we realise the hairy Deligne-Mumford DG-sheaf as the total object of the Cousin complex defined by a number of internal edges in stable graphs.
\par\medskip
For every $k\geq 0$ we introduce the following $HJ_g^{\circ}$-diagram:
\begin{equation}\label{fildi}
\mathcal {F}_{\geq k}\mathcal {HM}_g^{trop}\colon HJ_g^{\circ}\longrightarrow \textsf {Top}.
\end{equation}
Defined by a rule: a stratified topological space $\sigma_{\Gamma}^{hair}$ carries the following filtration:
$$
\sigma_{\Gamma}^{hair}:=e_{\geq 0}\sigma_{\Gamma}^{hair}\supset \dots \supset e_{\geq |E(\Gamma)|}\sigma_{\Gamma}^{hair},
$$
where $e_{\geq k}\sigma_{\Gamma}^{hair}$ is a subspace where at least $k$ "internal edge coordinates" are non-zero. Hence we define $\mathcal F_{\geq k} \mathcal M_{g}^{trop}$ by sending a stable graph $\Gamma$ to $e_{\geq k}\sigma_{\Gamma}^{hair}$ if $k\leq |E(\Gamma)|$ and otherwise to the empty set. We get the following decreasing filtration:
$$
\mathcal {HM}_g^{trop}:=\mathcal {F}_{\geq 0}\mathcal {HM}_g^{trop}\hookleftarrow \mathcal {F}_{\geq 1}\mathcal {HM}_g^{trop}\hookleftarrow\dots,
$$
where the morphism $j_{\geq k}\colon \mathcal {F}_{\geq k}\mathcal {HM}_g^{trop} \hookrightarrow \mathcal{HM}_g^{trop}$ is closed. We have the associated sequence of DG-combinatorial sheaves:
\begin{equation}\label{seq}
\begin{CD}
\EuScript {DM}_g^{hair}=j_{\geq 0!}j_{\geq 0}^*\EuScript {DM}_g^{hair}@<<<\dots @<<<j_{\geq k!}j_{\geq k}^*\EuScript {DM}_g^{hair} @<<<\dots
\end{CD}
\end{equation}
We have the Postnikov system associated with \eqref{seq} (cf. \cite{KaSh} and \cite{GK}). Hence for every $g\geq 2$ we realise $\EuScript {DM}_g^{hair}$ as the convolution object of the following complex of DG-constructible sheaves:
\begin{equation*}
\begin{CD}
j_{0!}j_{0}^*\EuScript {DM}_g^{hair} @>>> \dots @>{}>> j_{k!}j_{k}^*\EuScript {DM}_g^{hair}[k] @>{}>> \dots,
\end{CD}
\end{equation*}
where $j_k\colon \mathcal {F}_{k}\mathcal {HM}_g^{trop} \hookrightarrow \mathcal{HM}_g^{trop}$ is an open inclusion of diagrams. Here $\mathcal {F}_{k}\mathcal {HM}_g^{trop}$ is the open complement to $\mathcal {F}_{k+1}\mathcal {HM}_g^{trop}.$ We will show that the for every $k$ the following complex:

$$
\textsf{HGK}_{g,n}^{\hdot}:=R\Gamma_c(\mathcal {HM}_g^{trop},j_{n!}j_{n}^*\EuScript {DM}_g^{hair})[n]
$$
is contractible. Acting like in the proof of Theorem \ref{KD} the complex $\textsf{HGK}_{g,k}^{\hdot}$ can be realised as the total DG-vector space of a complex:
\begin{align}\label{fi}
\bigoplus_{\Gamma\in HJ^{k,0}_{g}}&\left(\bigotimes_{v\in V(\Gamma)} C^{\hdot}(\overline{\mathcal M}_{w(v),n_v},\mathbb Q)\otimes\mathrm {det}^h(\Gamma)\right)^{\mathrm {Aut}^h(\Gamma)}\rightarrow \dots \\
\dots &\rightarrow  \bigoplus_{\Gamma\in HJ^{k,n}_{g}}\left(\bigotimes_{v\in V(\Gamma)} C^{\hdot}(\overline{\mathcal M}_{w(v),n_v},\mathbb Q)\otimes\mathrm {det}^h(\Gamma)\right)^{\mathrm {Aut}^h(\Gamma)}\rightarrow\dots
\end{align}
Here ${\det}^h(\Gamma):=\textsf {or}_{S_{\Gamma}^{hair}}\cong \det(E(\Gamma))\otimes \det(m^{-1}(V(\Gamma)).$ Note that:
$$
C^{\hdot}(\overline{\mathcal M}_{\Gamma}/ \mathrm {Aut}^h(\Gamma),\varepsilon_p\otimes \epsilon_n)\cong\left(\bigotimes_{v\in V(\Gamma)} C^{\hdot}(\overline{\mathcal M}_{w(v),n_v},\mathbb Q)\otimes\mathrm {det}^h(\Gamma)\right)^{\mathrm {Aut}^h(\Gamma)}.
$$
Where by $\epsilon_n$ we have denoted a sign local system on the moduli stack $\overline{\mathcal M}_{\Gamma}/ \mathrm {Aut}^h(\Gamma)$ with a monodromy defined by a sign representation that acts on hairs of a stable graph $\Gamma.$ A differential in \eqref{fi} is defined by the following rule:
\par\medskip
Suppose that $\Gamma$ and $\Gamma'$ are two objects of $HJ_g$ such that $\Gamma$ has exactly $k$-hairs with the following property: if one contracts one of them the resulting graph will be isomorphic to $\Gamma'.$ We will denote this set of hairs by $K.$ Consider the following correspondence:
\begin{equation*}
\begin{diagram}[height=2.5em,width=2.5em]
(\overline{\mathcal M}_{\Gamma}\times K)/\mathrm {Aut}^{h}(\Gamma)&   \rTo^{\widetilde{\pi}_{\Gamma',\Gamma}}   &  \overline{\mathcal M}_{\Gamma'}/\mathrm {Aut}^{h}(\Gamma') &
\\
\dTo^{\pi_{\Gamma',\Gamma}} & &  && \\
\overline{\mathcal M}_{\Gamma}/\mathrm {Aut}^{h}(\Gamma)  &      &    \\
\end{diagram}
\end{equation*}
Where a morphism $\widetilde{\pi}_{\Gamma',\Gamma}$ is given by applying a pullback along the morphism $\pi_h$ at the hair $h\in K$ and $\pi_{\Gamma',\Gamma}$ is the projection morphism. Hence we get a differential defined by the rule:
\begin{equation}\label{Bdiff2}
\widetilde{\nabla}_1:=\sum_{\Gamma\in HJ^{k,n}_{g}\,,\Gamma'\in HJ_{g}^{k,n-1}\,\Gamma\twoheadrightarrow \Gamma'}
\pi_{\Gamma',\Gamma!}\circ\widetilde{\pi}_{\Gamma',\Gamma}^*
\end{equation}
Note that on a component $C^{\hdot}(\overline{\mathcal M}_{w(v),n_v},\mathbb Q)$ corresponding to the vertex $v\in V(\Gamma)$ with $m^{-1}(v)\neq \emptyset$ the differential $\widetilde{\nabla}_1$ acts by the following rule:
\begin{equation}\label{expc}
\widetilde{\nabla}_1\colon \omega\longmapsto \sum_{h\in m^{-1}(v)}(-1)^{h-1} \pi_h^*\omega
\end{equation}

\par\medskip

We shall construct an explicit homotopy and get the desired result. For every $k\geq 1$ we define the morphism
$$H\colon \textsf{HGK}_{g,k}^{\hdot}\longrightarrow \textsf{HGK}_{g,k}^{\hdot}[-1]$$
by the following rule:
\par\medskip
\begin{enumerate}[(i)]
  \item Suppose that $m^{-1}(v)=\emptyset,$ then we define this morphism to be zero on this component. 
\par\medskip
  \item Suppose that $m^{-1}(v)\neq \emptyset$ and let $\omega$ be an element of component $C^{\hdot}(\overline{\mathcal M}_{w(v),n_v},\mathbb Q),$ hence we set:
$$
H\colon \omega \longmapsto \sum_{h\in m^{-1}(v)}  \pi_{h!}(\omega \smile \widetilde{\psi}_{h})
$$
\end{enumerate}
Where:
$$\pi_{h!}\colon C^{\hdot}(\overline{\mathcal M}_{w(v),val_v+|m^{-1}(v)|},\mathbb Q)\longrightarrow C^{\hdot-2}(\overline{\mathcal M}_{w(v),val_v+|m^{-1}(v)\setminus h|},\mathbb Q)$$
is a Gysin pushforward morphism along the forgetful morphism $\pi_h$ and:
$$\widetilde{\psi}_h:=\frac{1}{2w(v)-2+n_v}\psi_h\in H^2(\overline{\mathcal M}_{w(v),val_v+|m^{-1}(v)|},\mathbb Q)$$
is a normalised psi-class at $h.$
Further applying \eqref{expc} together with a projection formula and the fact that the psi-class $\psi_h$ can be identified with the Euler class of the fibration $\pi_h$ one can explicitly compute (for details see Theorem $30$ in \cite{AWZ}): $$
(|m^{-1}(v)|+1)\omega=H(\widetilde{\nabla}_1(\omega))+\widetilde{\nabla}_1(H(\omega)),\qquad \omega \in C^{\hdot}(\overline{\mathcal M}_{w(v),n_v},\mathbb Q).$$ 
\end{proof}

Denote by $\textsf H_{\geq 1}\textsf{GK}_g^{\hdot}$ a decorated graph complex that computes the compactly supported cohomology of the $HJ_g^{\circ}$-diagram $\mathcal {H}_{\geq 1}\mathcal {M}^{trop}_{g}$ with coefficients in the combinatorial DG-sheaf $\EuScript {DM}_{g,\geq 1}^{hair}:=v^*_{\geq 1}\EuScript {DM}_{g}^{hair}:$
\begin{equation}\label{hgk}
\textsf H_{\geq 1}\textsf{GK}_g^{\hdot}:=R\Gamma_c(\mathcal {H}_{\geq 1}\mathcal {M}^{trop}_{g},\EuScript {DM}_{g,\geq 1}^{hair})
\end{equation}

We will call this complex a \textit{hairy Getzler-Kapranov complex with at least one hair} or just the hairy Getzler-Kapranov complex if it will not lead to confusion. We have the following immediate:

\begin{Cor}\label{cor} For $g>1$ the hairy Getzler-Kapranov complex $\textsf H_{\geq 1}\textsf{GK}_g^{\hdot}$ is quasi-isomorphic to $C_c^{\hdot}(\mathcal M_{g},\mathbb Q)[-1].$

\end{Cor}

\begin{proof} Consider the Gysin triangle associated with diagram \eqref{Diag2}:
$$
v_{\geq 1!}v^*_{\geq 1}\EuScript {DM}_{g}^{hair}\longrightarrow\EuScript {DM}_{g}^{hair}\longrightarrow v_{\emptyset*}v^*_{\emptyset}\EuScript {DM}_{g}^{hair}\longrightarrow v_{\geq 1!}v^*_{\geq 1}\EuScript {DM}_{g}^{hair}[1].
$$
Applying the compactly support cohomology functor and using Proposition \ref{van} we obtain that $R\Gamma_c(\mathcal {H}_{\geq 1}\mathcal {M}_{g}^{trop},v^*_{\geq 1}\EuScript {DM}_{g}^{hair})$ is quasi-isomorphic to a complex:
$$R\Gamma_c(\mathcal {H}_{\emptyset}\mathcal {M}_{g}^{trop},v^*_{\emptyset}\EuScript {DM}_{g}^{hair})[-1].$$ The latter complex can be identified with the shifted Getzler-Kapranov complex $\textsf{GK}_g[-1]^{\hdot}$ and hence by Proposition \ref{KD} with the DG-vector space $C_c^{\hdot}(\mathcal M_{g},\mathbb Q)[-1].$
\end{proof}

\begin{remark} It is possible to extend the definition of the hairy Getzler-Kapranov complex $\textsf H_{\geq 1}\textsf {GK}_g^{\hdot},$ to the case when $g=1.$ One can consider a category $H_{\geq 1}J_g$ with objects being stable marked graphs of genus $g$ with \textit{at least one marking}. Note that we have a natural inclusion of categories $H_{\geq 1}J_g\hookrightarrow HJ_g.$ The category $H_{\geq 1}J_g$ is well defined for $g=1$ and one can consider a $H_{\geq 1}J_g^{\circ}$-diagram $\mathcal H_{\geq 1}\mathcal M_g^{trop}.$ For $g=1$ we can define the hairy Getzler-Kapranov complex $\textsf H_{\geq 1}\textsf {GK}_1$ verbatim to \eqref{hgk}. Analogous to the proof of Proposition \ref{van} one can compute the cohomology of $\textsf H_{\geq 1}\textsf {GK}_1^{\hdot}$ (see Theorem \ref{th}).

\end{remark}

\subsection{The Willwacher differential $\nabla_1$} In this subsection, we will relate the hairy Getzler-Kapranov complex $\textsf H_{\geq 1}\textsf{GK}_g$ to the certain double complex which comes from varying a number of markings on the moduli stacks $\mathcal M_{g,n}.$ To do it let us recall some definitions:
\par\medskip
For $g>0$ and a finite set $I$ such that $2g+|I|-2> 0$ we denote by $\mathcal M_{g,I}^{rt}$ \textit{the moduli stack of $I$-marked curves of genus $g$ with rational tails}. For $g\geq 2$ this stack can be defined as a fibered product:\footnote{Morphisms in the corresponding fibered square are given by the canonical open inclusion $j_I$ and the forgetful morphism $\pi\colon \overline{\mathcal M}_{g,I}\longrightarrow \overline{\mathcal M}_g$ \eqref{forg}.}
\begin{equation}
\mathcal M_{g,I}^{rt}:= \overline{\mathcal M}_{g,I}\times_{\overline{\mathcal M}_{g} } \mathcal M_{g}
\end{equation}
For $g=1$ the stack $\mathcal M_{g,I}^{rt}$ can be defined as the \textit{moduli stack of curves with compact Jacobian} $\mathcal M_{g,n}^{c}$ i.e. is consists of nodal curves with a dual graph being a tree. By the definition we have $\mathcal M_{g,1}^{rt}\cong \mathcal M_{g,1}.$ By the construction for every $x$ we have a proper morphism of stacks:
\begin{equation}
\mu_x\colon \mathcal M_{g,I\cup \{x\}}^{rt}\longrightarrow \mathcal M_{g,I}^{rt}.
\end{equation}
We obtain the following symmetric $\Delta$-stack $\mathcal M_{g,\hdot}^{rt}$\footnote{By the symmetric $\Delta$-stack we understand a functor from a category of symmetric semi-simplicial sets $I^{op}$ to the category of Deligne-Mumford stack with proper morphisms.}:
\begin{equation}
\begin{diagram}[height=2.5em,width=4em]
\mathcal M_{g,1}^{rt} &    \pile{\lTo^{}\\  \\ \lTo_{}}   &  \mathcal M_{g,2}^{rt}  &  \pile{\lTo_{}\\  \\ \lTo_{}\\ \\ \lTo_{}} & \mathcal M_{g,3}^{rt} & \dots & \\
\end{diagram}
\end{equation}
Applying the functor of the cochains with compact support we obtain the symmetric $\Delta$-cochain complex. By a version of the Dold-Puppe construction we construct the following DG-vector space:
\begin{equation}\label{rt1}
\begin{CD}
C_c^{\hdot}(\mathcal M_{g,1}^{rt},\mathbb Q) @>{\nabla_1^{rt}}>> \dots @>{\nabla_1^{rt}}>> \left(C_c^{\hdot}(\mathcal M_{g,n}^{rt},\mathbb Q)\otimes_{\Sigma_n} \mathrm {sgn}_n\right)^{\Sigma_n} @>{\nabla_1^{rt}}>> \dots
\end{CD}
\end{equation}
Where a morphism $\nabla_1^{rt}$ is defined by the following rule. For every $n\geq 1$ denote by $\mathcal M_{\bullet_{g,n}}^{rt}$ the following stack:
$$
\mathcal M_{\bullet_{g,n}}^{rt}:=\mathcal M_{g,n}^{rt} / \mathrm {Aut}^h(\bullet_{g,n})
$$
Then one can consider the following correspondence:
\begin{equation*}
\begin{diagram}[height=2.5em,width=2.5em]
\mathcal M_{\bullet_{g,n},\bullet_{g,n+1}}^{rt}&   \rTo^{\widetilde{\mu}_{\bullet_{g,n},\bullet_{g,n+1}}}   &  \mathcal M_{\bullet_{g,n}}^{rt} &  \\
\dTo^{\pi_{\bullet_{g,n},\bullet_{g,n+1}}} & &  && \\
\mathcal M_{\bullet_{g,n+1}}^{rt}  &      &    \\
\end{diagram}
\end{equation*}
Where a stack $\mathcal M_{\bullet_{g,n},\bullet_{g,n+1}}^{rt}$ is defined as $\big(\mathcal M_{g,n+1}^{rt}\times m^{-1}(V(\bullet_{g,n+1})\big)/ \Sigma_{n+1}$ and $\pi_{\bullet_{g,n},\bullet_{g,n+1}}$ is the projection morphism. $\mu_{\bullet_{g,n},\bullet_{g,n+1}}$ is defined as a morphism that factorises the morphism $\big(\mathcal M_{g,n+1}^{rt}\times m^{-1}(V(\bullet_{g,n+1})\big)\longrightarrow \mathcal M_{\bullet_{g,n}}^{rt}.$ The later morphism is defined by applying the morphism $\mu_i$ along the $i$-hair of $\bullet_{g,n+1}.$ We denote by $\epsilon_n$ the sign local system on the stack $\mathcal M_{\bullet_{g,n}}^{rt}$ with the monodromy defined by the sign representation of $\mathrm {Aut}^{h}(\bullet_{g,n}):=\Sigma_n.$
We define the following pull-push morphism of DG-vector spaces (cf. \eqref{Bdiff2}):
$$
\nabla_1^{rt}\colon C_c^{\hdot}(\mathcal M_{\bullet_{g,n}}^{rt},\epsilon_n)\longrightarrow C^{\hdot}_c(\mathcal M_{\bullet_{g,n+1}}^{rt},\epsilon_{n+1})$$
By the rule:
$$
\nabla_1^{rt}:=\pi_{\bullet_{g,n},\bullet_{g,n+1}!}\mu^*_{\bullet_{g,n},\bullet_{g,n+1}}
$$
Note the cochains with compact support of $\mathcal M_{\bullet_{g,n}}^{rt}$ with coefficients in the sign local system can be naturally identified with the DG-vector space $(C_c^{\hdot}(\mathcal M_{g,n}^{rt},\mathbb Q)\otimes_{\Sigma_n} \mathrm {sgn_n})^{\Sigma_n}$ and hence we get the desired operator.
\par\medskip
The moduli stack $\mathcal M_{g,n}^{rt}$ is equipped with a natural inclusion $j^{rt}\colon \mathcal M_{g,n}\hookrightarrow  \mathcal M_{g,n}^{rt}.$ Then we have the following:
\begin{Prop}\label{quasirt} For $2g+n-2>0$ a $!$-extension functor along $j^{rt}$ induces the quasi-isomorphism of complexes which computes the compactly supported cohomology with coefficients in the sign local systems:
\begin{equation}\label{rt}
j_!^{rt}\colon C_c^{\hdot}(\mathcal M_{g,n}/\Sigma_n,\epsilon_n) \overset{\sim}{\longrightarrow}C_c^{\hdot}(\mathcal M_{g,n}^{rt}/\Sigma_n,\epsilon_n)
\end{equation}

\end{Prop}

\begin{proof} Denote by $D$ the complement to $\mathcal M_{g,I}$ in $\mathcal M_{g,I}^{rt}.$ Then $D$ is a divisor with the following decomposition:
$$
D=\coprod D_{T},
$$
where $T\in HJ_g$ has a combinatorial type of a tree with a unique vertex of genus $g$ and all other vertices are marked by genus $0.$ We claim that a complex $C^{\hdot}_c(D / \Sigma_n,\epsilon_{n})$ which computes the compactly supported cohomology with coefficients in the sign local system is acyclic:
\par\medskip
A stable tree $T$ always has a vertex $v$ with the associated moduli space $\overline{\mathcal M}_{w(v),n_v}=\overline{\mathcal M}_{0,|m^{-1}(v)|+1},$ where $2\leq |m^{-1}(v)|\leq n$ is a number of hairs attached to $v$ and $1$ represents a unique root (an internal half-edge) attached to $v.$ Hence the claim will follow from the following:
\begin{Lemma}\label{vc0} For every $n>1$ we have an equivalence:
$$
H^{\hdot}(\overline{\mathcal M}_{0,n+1}/\Sigma_n,\epsilon_n)=0
$$
\end{Lemma}
\begin{proof} We will prove this statement by induction. Assume that the assertion is true for $k.$ We will prove it for $k+1:$
\par\medskip
Diagram \eqref{Diag1} induces the Gysin long exact sequence:
\begin{equation}
\dots\rightarrow H^{i-1}(\partial \mathcal M_{0,k+1}/\Sigma_k,\epsilon_k)\rightarrow H^i_c(\mathcal M_{0,k+1}/\Sigma_k,\epsilon_k)\rightarrow H^i_c(\overline{\mathcal M}_{0,k+1}/\Sigma_k,\epsilon_k)\rightarrow\dots
\end{equation}
First, we will show that the compactly supported cohomology of the smooth locus vanishes:
\begin{equation}\label{vc1}
H^{\hdot}_c(\mathcal M_{0,k+1}/\Sigma_k,\epsilon_k)=0
\end{equation} 
Namely due to \cite{VAS1} we know that $H^{\hdot}_c(\mathcal M_{0,k}/\Sigma_k,\epsilon_k)=0.$ We consider the fibration:
$$
\pi_1\colon \mathcal M_{0,k+1}/\Sigma_k\longrightarrow  \mathcal M_{0,k}/\Sigma_k
$$
Bu the Serre spectral sequence we get the vanishing of \eqref{vc1}.
\par\medskip 
Hence it is enough to prove that $H^{\hdot}(\partial \mathcal M_{0,k+1}/\Sigma_k,\epsilon_k)$ vanishes as well. Consider the Deligne complex which computes $H^{\hdot}(\partial \mathcal M_{0,k+1}/\Sigma_k,\epsilon_k).$ The $p$-cochains of this complex are defined by the following rule:
\begin{equation}\label{g0}
 \bigoplus_{T\in HJ^{p,k+1}_{0}}\left(\bigotimes_{v\in V(T)} C^{\hdot}(\overline{\mathcal M}_{w(v),n_v},\mathbb Q)\otimes\mathrm {det}^h(T)\right)^{\mathrm {Aut}^{h,1}(T)}.
 \end{equation}
Where $T$ is a stable tree with $p$ edges and $k+1$ leaves and $\mathrm {Aut}^{h,1}(T)$ is a group of automorphisms of a tree that fix one marking. This group of automorphisms consists of permutations of hairs of a tree. Each tree $T$ has a vertex $v'$ such that all hairs $m^{-1}(v')$ attached to this vertex are unfrozen i.e. all hairs are allowed to be permuted. Denote by $\mathrm {Aut}(T)^{v'}$ the subgroup of $\mathrm {Aut}^{h,1}(T)$ which permute these hairs. Note that $|m^{-1}(v')|<k$ and hence by the assumption of induction we have: 
$$H^{\hdot}(\overline{\mathcal M}_{0,|m^{-1}(v')|+1}/\Sigma_{|m^{-1}(v')|},\epsilon_{|m^{-1}(v')|})=0.$$
Since the moduli stack of stable rational curves is formal and each element in $\eqref{g0}$ must be in particular skew-invariant under $\mathrm {Aut}^{h,1}(T)$ and we get the vanishing of \eqref{g0}.

\end{proof} 

From the standard Gysin long exact sequence:
\begin{equation}
\dots\rightarrow H^{i-1}_c(D/\Sigma_n,\epsilon_n)\rightarrow H^i_c(\mathcal M_{g,n}/\Sigma_n,\epsilon_n)\rightarrow H^i_c(\mathcal M_{g,n}^{rt}/\Sigma_n,\epsilon_n)\rightarrow\dots
\end{equation}
we obtain the desired quasi-isomorphism.

\end{proof}

\begin{remark}[V. Dotsenko] To prove Lemma \ref{vc0} one can use the following argument. Consider the hypercommutative operad $\mathcal {H}ycom=\{H_{\hdotc}(\overline{\mathcal M}_{0,k+1})\}_{k \geq 2}.$ It is the quotient of the free operad generated by elements $\{m_k\}_{k \geq 2}$ (fundamental classes) modulo the quadratic ideal. A sign representation does not appear in the decomposition of the free operad into irreducible representations and thus operadic ideal is this subrepresentation, hence we do not get any new irreducible representations after passing to the quotient. 

\end{remark}

\begin{Def} We define the \textit{Willwacher differential} $\nabla_1$\footnote{We named this object in honour of Thomas Willwacher who predicted its existence \cite{WillL}.} as a unique zigzag morphism:
$$
\nabla_1\colon C_c^{\hdot}(\mathcal M_{g,n}/ \Sigma_{n},\epsilon_n) \longrightarrow  C_{c}^{\hdot}(\mathcal M_{g,n+1}/ \Sigma_{n+1},\epsilon_{n+1}),
$$
making the following diagram commutes: 
\begin{equation*}
\begin{diagram}[height=3.5em,width=4.2em]
 C_c^{\hdot}(\mathcal M_{\bullet_{g,n}}^{rt},\epsilon_n) & &  \rTo_{}^{\nabla_1^{rt}} &  &  C_{c}^{\hdot}(\mathcal M_{\bullet_{g,n+1}}^{rt},\epsilon_{n+1}) &  \\
\uTo_{\sim}^{j^{rt}_!} & & &  & \uTo_ {\sim}^{j_!^{rt}}  && \\
C_c^{\hdot}(\mathcal M_{g,n}/ \Sigma_{n},\epsilon_n) & &  \rDotsto^{\nabla_1^{rt}} &  &  C_{c}^{\hdot}(\mathcal M_{g,n+1}/ \Sigma_{n+1},\epsilon_{n+1}) & \\
\end{diagram}
\end{equation*}

\end{Def}

Thus one may consider the following cochain complex in the derived category of vector spaces:
\begin{equation}\label{dc}
\begin{CD}
C_c^{\hdot}(\mathcal M_{g,1},\mathbb Q) @>{\nabla_1}>> \dots @>{\nabla_1}>> \left(C_c^{\hdot}(\mathcal M_{g,n},\mathbb Q)\otimes_{\Sigma_n} \mathrm {sgn}_n\right)^{\Sigma_n} @>{\nabla_1}>> \dots
\end{CD}
\end{equation}
Now we can relate the cohomology of the hairy Getzler-Kapranov complex to the cohomology of the total DG-vector space of the complex above:

\begin{Prop}\label{mp} For every $g\geq 1$ the hairy Getzler-Kapranov complex $\textsf H_{\geq 1}\textsf{GK}_g$ is a quasi-isomorphic to the total DG-vector space of \eqref{dc}.
\end{Prop}

\begin{proof} We are going to use the Cousin resolution for the hairy Deligne-Mumford DG-sheaf $\EuScript {DM}_{g,\geq 1}^{hair}$ (a filtration is defined by a number of non-internal edges (hairs) in the stable graph):
\par\medskip
For every $g\geq 1$ consider the following sequence of $H_{\geq 1}J_g^{\circ}$-diagrams:
$$
\mathcal H_{\geq 1} \mathcal M_{g}^{trop}\hookleftarrow\dots\hookleftarrow\mathcal H_{\geq k} \mathcal M_{g}^{trop}\hookleftarrow\dots
$$
Here $\mathcal H_{\geq k} \mathcal M_{g}^{trop}$ is an $H_{\geq 1}J_g^{\circ}$-diagram defined by the following rule. A stratified topological space $h_{\geq 1}\sigma_{\Gamma}^{hair}$ carries the following filtration:
$$
h_{\geq 1}\sigma_{\Gamma}^{hair}\supset \dots \supset h_{\geq |m^{-1}(V(\Gamma))|}\sigma_{\Gamma}^{hair},
$$
where $h_{\geq k}\sigma_{\Gamma}^{hair}$ is a subspace where at least $k$ "hair coordinates" are non-zero. Hence $\mathcal H_{\geq k} \mathcal M_{g}^{trop}$ sends a stable graph $\Gamma$ to 
$h_{\geq k}\sigma_{\Gamma}^{hair}$ if $k\leq |m^{-1}(V(\Gamma))|$ and to the empty set otherwise. We have the associated sequence of morphisms of constructible DG-sheaves:
\begin{equation}
\begin{CD}
\EuScript {DM}_{g,\geq 1}^{hair} @>>>\dots @>>> i_{\geq k !}i^{*}_{\geq k}\EuScript {DM}_{g,\geq 1}^{hair} @>>>\dots
\end{CD}
\end{equation}
Where $i_{\geq k}\colon \mathcal H_{\geq k} \mathcal M_{g}^{trop}\hookrightarrow \mathcal H_{\geq 1} \mathcal M_{g}^{trop}$ is a closed inclusion. We have the standard Postnikov system associated with the sequence above. Hence we can realise the combinatorial DG-sheaf $\EuScript {DM}_{g,\geq 1}^{hair}$ as the total object of the following complex of constructible DG-sheaves:
\begin{equation}
\begin{CD}
i_{1 !}i_{1}^*\EuScript {DM}_{g,\geq 1}^{hair} @>>>\dots @>>> i_{k!}i_{k}^*\EuScript {DM}_{g,\geq 1}^{hair}[k-1] @>>> \dots
\end{CD}
\end{equation}
Where $i_k\colon \mathcal H_{k} \mathcal M_{g}^{trop}\hookrightarrow \mathcal H_{\geq 1} \mathcal M_{g}^{trop}$ is the open inclusion and $\mathcal H_{k} \mathcal M_{g}^{trop}$ is the complement to $\mathcal H_{\geq k+1} \mathcal M_{g}^{trop}.$ Passing to the compactly supported cochains we get the following DG-vector space:
\begin{equation}\label{h1}
\begin{CD}
\textsf {H}_{1}\textsf {GK}_g^{\hdot} @>{\widetilde{\nabla}_1}>> \dots @>{\widetilde{\nabla}_1}>> \textsf {H}_{k}\textsf {GK}_g^{\hdot} @>{\widetilde{\nabla}_1}>> \dots
\end{CD}
\end{equation}
Where we have used the following notation:
\begin{equation}\label{hgke}
\textsf {H}_{k}\textsf {GK}_g^{\hdot}:=R\Gamma_c(\mathcal H_{\geq 1}\mathcal M^{trop}_g,i_{k !}i_{k}^*\EuScript {DM}_{g,\geq 1}^{hair}[k-1])
\end{equation}
Note that \eqref{hgke} is a decorated graph complex that consists of genus $g$ decorated stable graphs with exactly $k$ hairs and $\widetilde{\nabla}_1$ is a differential defined by \eqref{Bdiff2}.

Acting like in the proof of Proposition \ref{KD} for every $n$ we can explicitly realise the complex $\textsf {H}_{n}\textsf {GK}_g^{\hdot}$ as a total DG-vector space associated with double complex:
\begin{align}
C^{\hdot}(\overline{\mathcal M}_{g,n},\mathbb Q)&\otimes {\det}^h(\bullet_{g,n})\rightarrow \dots \\
\dots &\rightarrow  \bigoplus_{\Gamma\in HJ^{k,n}_{g}}\left(\bigotimes_{v\in V(\Gamma)} C^{\hdot}(\overline{\mathcal M}_{w(v),n_v},\mathbb Q)\otimes\mathrm {det}^h(\Gamma)\right)^{\mathrm {Aut}^h(\Gamma)}\rightarrow\dots,
\end{align}
The differential $\widetilde{D}$ in the complex above is defined by \eqref{diffe}. By Proposition \ref{D1} we obtain that the complex $\textsf{H}_n\textsf{GK}_g^{\hdot}$ is quasi-isomorphic to $C_c^{\hdot}(\mathcal M_{g,n}/\Sigma_n,\epsilon_n).$ 
In terms of P. Deligne's complexes morphism \eqref{rt} corresponds to pullbacks along natural inclusions of the corresponding strata.\footnote{The compactly supported cochains of the moduli stack $\mathcal M_{g,n}^{rt}/\Sigma_n$ can be realised by P. Deligne's complex (we assume that this moduli stack is compactified by the stack $\overline{\mathcal M}_{g,n}/ \Sigma_n,$ with the corresponding complement $\mathcal D^{rt}$).} This follows from the base change for the compactly supported cochains. Hence the following diagram commutes:
\begin{equation*}
\begin{diagram}[height=3.1em,width=4.4em]
C_c^{\hdot}(\mathcal M_{g,n}/\Sigma_{n},\epsilon_n) & \rTo^{j^{rt}_!}_{\sim}   &  C_c^{\hdot}(\mathcal M_{g,n}^{rt}/\Sigma_{n},\epsilon_n)  \\
\dTo^{\widetilde{\nabla}_1} & & \dTo^{{\nabla}_1^{rt}}   \\
C_c^{\hdot}(\mathcal M_{g,n+1}/\Sigma_{n+1},\epsilon_{n+1}) & \rTo^{j_!^{rt}}_{\sim}   &  C_c^{\hdot}(\mathcal M_{g,n+1}^{rt}/\Sigma_{n+1},\epsilon_{n+1}) \\
\end{diagram}
\end{equation*}
By the definition of the Willwacher differential, the morphism $\widetilde{\nabla}_1$ coincides with $\nabla_1$ and the desired result follows.

\end{proof}

From Proposition \ref{mp} and Theorem \ref{wgk} we immediately get:
\begin{Cor} For any $g,n$ such that $2g+n-2>0$ the induced Willwacher differential:
$$
\nabla_1\colon \left(H^{\hdot}_c(\mathcal M_{g,n},\mathbb Q)\otimes_{\Sigma_n} \mathrm{sgn}_n\right )^{\Sigma_n}\longrightarrow\left(H^{\hdot}_c(\mathcal M_{g,n+1},\mathbb Q)\otimes_{\Sigma_{n+1}} \mathrm{sgn}_{n+1}\right )^{\Sigma_{n+1}}
$$
preserves the canonical weight quotients. 
\end{Cor}

Our main result in this section is the following theorem which was originally conjectured in \cite{AWZ} (Conjecture $31$):
\begin{Th}\label{th} Consider the following complex:
\begin{equation}\label{brc1}
\begin{CD}
C_c^{\hdot}(\mathcal M_{g,1},\mathbb Q) @>{\nabla_1}>> \dots @>{\nabla_1}>> \left(C_c^{\hdot}(\mathcal M_{g,n},\mathbb Q)\otimes_{\Sigma_n} \mathrm {sgn}_n\right)^{\Sigma_n} @>{\nabla_1}>> \dots
\end{CD}
\end{equation}
\begin{enumerate}[(i)]
\item For $g\geq 2$ \eqref{brc1} is quasi-isomorphic to $C_c^{\hdot-1}(\mathcal M_{g},\mathbb Q).$ 
\par\medskip 
\item For $g=1$ the cohomology of \eqref{brc1} is given by:
$$
\prod_{n=3}^{\infty}(S_{n+1}\oplus \overline{S}_{n+1}\oplus Eis_{n+1})[2n]\oplus \mathbb Q[3]
$$\par\medskip 
\item For $g=0$ \eqref{brc1} is acyclic. 
\end{enumerate} 
\end{Th}

\begin{proof}

\begin{enumerate}[(i)]
\item For $g\geq 2$ the proof immediately follows from Proposition \ref{mp} and Corollary \ref{cor}.
\item Consider a local system $\mathbb V_1$ on $\mathcal M_{1,1}$ which is defined by the rule $\mathbb V_1:=R^1\pi_{1\,*}\mathbb Q.$ By $\mathbb V_{k}$ we denote the $k$-symmetric power of this local system. Note that for every $k$ the local system $\mathbb V_k$ underlies the certain VPHS of the weight $k,$ which will be denoted by the same symbol. Due to the Eichler-Shimura theory, the weight filtration $W_{\hdotc}$ on the cohomology of the VPHS has the following description:
$$
W_0H_c^1(\mathcal M_{1,1},\mathbb V_k) \subset W_{k+1}H_c^1(\mathcal M_{1,1},\mathbb V_k):=H_c^1(\mathcal M_{1,1},\mathbb V_k),
$$
with the following graded quotients: 
\begin{align*}
W_0H_c^1(\mathcal M_{1,1},\mathbb V_k)&\cong Eis_{k+2},\\
\mathrm {Gr}_W^{k+1}H_c^1(\mathcal M_{1,1},\mathbb V_k)&\cong S_{k+2}\oplus \overline{S}_{k+2}.
\end{align*}
Here $S_{k+2}$ is the vector space of cusp forms of the weight $k+2$ and $\overline {S}_{k+2}$ is the vector space of the antiholomorphic cusp forms of weight $k+2$ and $Eis_{k+2}$ is the vector space of the Eisenstein series of weight $k+2.$ From Proposition $1$ from \cite{CF} (see also \cite{Gor} and \cite{Pet2}) one can show that for any $n>1$ there is an isomorphism of graded vector spaces:
$$(H^{\hdot}_c(\mathcal M_{1,n},\mathbb Q)\otimes_{\Sigma_n} \mathrm {sgn}_n)^{\Sigma_n}=H^{1}_c(\mathcal M_{1,1},\mathbb V_{n-1} )[n]$$
Further notice that the "fundamental cohomology class" $\omega\in H^2_c(\mathcal M_{1,1},\mathbb Q)$ goes to zero under the differential $\nabla_1$ (due to the lack of non-zero modular forms for odd weight). Hence the cohomology of \eqref{brc1} is isomorphic to (cf. \cite{Tael} \cite{Beh}):
$$
\prod_{n=3}^{\infty}(S_{n+1}\oplus \overline{S}_{n+1}\oplus Eis_{n+1})[2n]\oplus \mathbb Q[3]
$$
\item Following \cite{VAS1} we have $(H_c^{\hdot}(\mathcal M_{0,n},\mathbb Q)\otimes \mathrm {sgn}_n)^{\Sigma_n}=0,$ hence the result follows. 

\end{enumerate}
\end{proof}

\bibliographystyle{amsalpha}
\bibliography{tt}
\end{document}